\documentclass{amsart}
\usepackage{amsmath,amssymb,amsthm,graphics,graphicx,color}

\usepackage[breaklinks,colorlinks=false]{hyperref}
\usepackage{subfigure}

\def\abs #1{\lvert #1\rvert}

\newcommand{\beq}{\begin{equation}}
\newcommand{\eeq}{\end{equation}}

\newcommand{\comment}[1]{\emph{\color{red}Comment:\color{black} #1}}
\newlength{\commentslength}
\newcommand{\comments}[1]{
\hspace{-2\parindent}
\addtolength{\commentslength}{-\commentslength}
\addtolength{\commentslength}{\linewidth}
\addtolength{\commentslength}{-\parindent}
\fcolorbox{red}{white}{\smallskip\begin{minipage}[c]{\commentslength}
\emph{Comments:}\begin{itemize}#1\end{itemize}\end{minipage}}\bigskip
}
\renewcommand{\comment}[1]{}\renewcommand{\comments}[1]{}

\begin{document}

\theoremstyle{plain}
\newtheorem{conjecture}{Conjecture}[section]
\newtheorem{theorem}[conjecture]{Theorem}
\newtheorem*{Theorem}{Theorem}
\newtheorem{corollary}[conjecture]{Corollary}
\newtheorem*{Corollary}{Corollary}
\newtheorem{lemma}[conjecture]{Lemma}
\newtheorem{claim}[conjecture]{Claim}
\newtheorem{proposition}[conjecture]{Proposition}
\newtheoremstyle{note}{}{}{\slshape}{}{\bfseries}{.}{ }{}
\theoremstyle{remark}
\newtheorem{remark}[conjecture]{Remark}
\theoremstyle{definition}
\newtheorem{definition}[conjecture]{Definition}
\newtheorem*{openquestion}{Open Question}
\newtheorem*{notation}{Notation}

\newcommand{\thmref}[1]{\hyperref[#1]{{Theorem~\ref*{#1}}}}
\newcommand{\lemref}[1]{\hyperref[#1]{{Lemma~\ref*{#1}}}}
\newcommand{\corref}[1]{\hyperref[#1]{{Corollary~\ref*{#1}}}}
\newcommand{\propref}[1]{\hyperref[#1]{{Proposition~\ref*{#1}}}}
\newcommand{\eqnref}[1]{\hyperref[#1]{{(\ref*{#1})}}}
\newcommand{\defref}[1]{\hyperref[#1]{{Definition~\ref*{#1}}}}
\newcommand{\secref}[1]{\hyperref[#1]{{Section~\ref*{#1}}}}
\newcommand{\figref}[1]{\hyperref[#1]{{Figure~\ref*{#1}}}}
\newcommand{\claimref}[1]{\hyperref[#1]{{Claim~\ref*{#1}}}}

\newenvironment{acknowledgements}{\bigskip\noindent{\bf Acknowledgements. }}{}

\allowdisplaybreaks[1]

\sloppy

% headers

\author{Ben W. Reichardt}
\address{Institute for Quantum Information\\
California Institute of Technology\\
MC 107-81~-IQI\\
Pasadena, CA 91125-8100}
\email{breic@caltech.edu}

\title{Proof of the Double Bubble Conjecture in $\bf{R^n}$}


\begin{abstract}
The least-area hypersurface enclosing and separating two given volumes in $\bf{R^n}$ is the standard double bubble.
\end{abstract}

\maketitle

\section{Introduction}

\subsection{The Double Bubble Conjecture}

We extend the proof of the double bubble theorem~\cite{HutchingsMorganRitoreRos00dblbblR3} from $\bf{R^3}$ to $\bf{R^n}$.  

\begin{theorem}[Double Bubble Conjecture] \label{t:conjecture}
The least-area hypersurface enclosing and separating two given volumes in $\bf{R^{n}}$ is the standard double soap bubble of \figref{f:standarddoublebubble}, consisting of three ($n-1$)-dimensional spherical caps intersecting at $120$ degree angles.  (For the case of equal volumes, the middle cap is a flat disk.)  
\end{theorem}

\begin{figure}
\includegraphics[scale=.175]{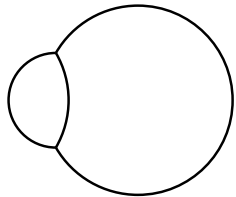}
\caption{The standard double bubble, consisting of three spherical caps meeting at $120$ degree angles, is now known to be the least-area hypersurface that encloses two given volumes in $\bf{R^{n}}$.} \label{f:standarddoublebubble}
\end{figure}

In 1990, Foisy, Alfaro, Brock, Hodges and Zimba \cite{FoisyAlfaroBrockHodgesZimba93dblbblR2} proved the Double Bubble Conjecture in $\bf{R^{2}}$.  
In 1995, Hass, Hutchings and Schlafly \cite{HassHutchingsSchlafly95dblbbl, HassSchlafly00dblbbl} used a computer to prove the conjecture for the case of equal volumes in $\bf{R^{3}}$.  

Arguments since have relied on the Hutchings structure theorem (\thmref{t:structure}), stating roughly that the only possible nonstandard minimal double bubbles are rotationally symmetric about an axis and consist of ``trees" of annular bands wrapped around each other~\cite{Hutchings97structure}.  Figures~\ref{f:structure} and~\ref{f:graphstructure} show examples of ``$4+4$" bubbles, in which the region for each volume is divided into four connected components.  

\begin{figure}
\includegraphics[scale=.28]{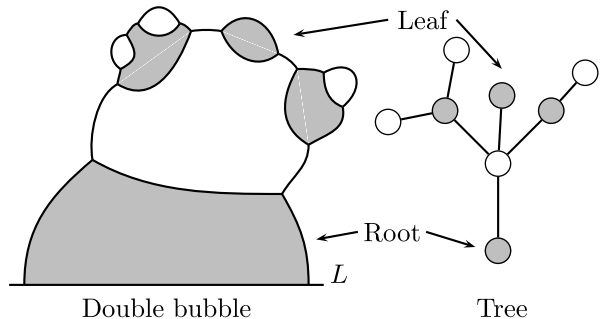}
\caption{A nonstandard minimal double bubble must be a hypersurface of revolution about an axis $L$, consisting of a central bubble with layers of toroidal bands.  Here we show the generating curves of a typical $4+4$ double bubble, together with the associated tree $T$.} \label{f:structure}
\end{figure}

\begin{figure} 
\includegraphics[scale=.28]{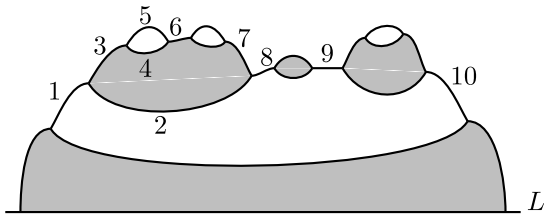}
\caption{The generating curves for another possible $4+4$ bubble with the same tree structure as in \figref{f:structure}.} \label{f:graphstructure}
\end{figure}

In 2000, Hutchings, Morgan, Ritor\'{e} and Ros \cite{HutchingsMorganRitoreRos00dblbblR3} used stability arguments to prove the conjecture for all cases in $\bf{R^{3}}$.  For $\bf{R^3}$, component bounds after Hutchings \cite{Hutchings97structure} guarantee that the region enclosing the larger volume is connected and the region enclosing the smaller volume has at most two components.  (For equal volumes, both regions need be connected.)  Eliminating ``$1 + 2$" and ``$1+1$" nonstandard bubbles proved the conjecture in $\bf{R^3}$.  \secref{s:instability} below sketches their instability argument.  (\cite{Morgan00geometry} also gives background and a proof sketch.)  

In 2003, Reichardt, Heilmann, Lai and Spielman \cite{ReichardtHeilmannLaiSpielman03dblbblR4} extended their arguments to $\bf{R^4}$, and to higher dimensions $\bf{R^n}$ provided one volume is more than twice the other.  In these cases, component bounds guarantee that one region is connected and the other has a finite number $k$ of components.  Eliminating $1 + k$ bubbles proved the conjecture in these cases.

The same component bounds show that in $\bf{R^5}$ it suffices to eliminate $2+2$ bubbles (as well as $1+k$ bubbles) to prove the Double Bubble Conjecture.  In $\bf{R^6}$ it suffices to eliminate also $2+3$ bubbles.  However, in $\bf{R^n}$ generally we know only that the larger region has at most three components and the smaller region has a finite number of components~\cite{HeilmannLaiReichardtSpielman99small}.  

Here, we extend the methods of Hutchings \emph{et al.} and Reichardt {\em et al.} to prove the Double Bubble Conjecture in $\bf{R^n}$ for $n \geq 3$ for arbitrary volumes.  We prove that $j + k$ nonstandard bubbles are not minimizing for arbitrary finite component counts $j, k$.  The arguments are similar in spirit to those of \cite{ReichardtHeilmannLaiSpielman03dblbblR4}, except we take advantage of more properties of constant-mean-curvature surfaces of revolution in order to eliminate previously problematic cases.  Our arguments also simplify the previous proofs in $\bf{R^3}$ and $\bf{R^4}$ because they eliminate the need for component bounds.

\subsection{The Instability Argument}  \label{s:instability}

An area-minimizing double bubble $\Sigma$ exists and has an axis of rotational symmetry $L$.  Assume that $\Sigma$ is a nonstandard double bubble.  
Consider small rotations about a line $M$ orthogonal to $L$, chosen so that the points of tangency between $\Sigma$ and the rotation vectorfield $v$ separate the bubble into at least four pieces, as in \figref{f:separatingex}.  Then we can linearly combine the restrictions of $v$ to each piece to obtain a vectorfield that vanishes on one piece and preserves volume.  By regularity for eigenfunctions, $v$ is tangent to certain related parts of $\Sigma$, implying that they are spheres centered on $L \cap M$.  In turn, this implies that there are too many spherical pieces of $\Sigma$, leading to a contradiction.  
This is the instability argument of \cite{HutchingsMorganRitoreRos00dblbblR3} behind \thmref{t:separation}.  

\begin{figure}
% 36 pt font figure should be scaled by .28
% 22 pt font figure should be scaled by .46
\includegraphics[scale=.46]{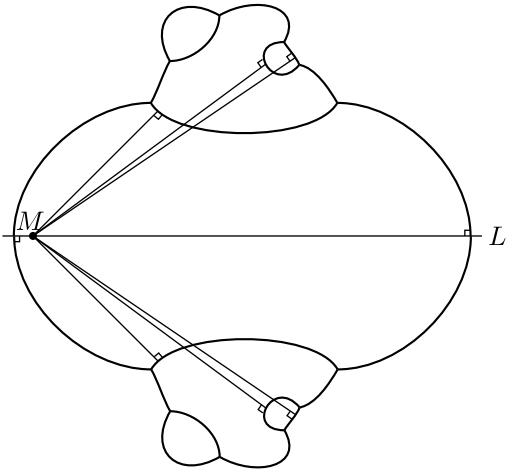}
\caption{The lines orthogonal to $\Sigma$ through the points of the separating set all pass through $M$.  $\Sigma$ cannot be a minimizer.} \label{f:separatingex}
\end{figure}

Therefore, no such useful perturbation axis $M$ can exist.  
By induction, starting at the connected components corresponding to leaves in the tree of \figref{f:structure} we classify all possible configurations in which no such $M$ can be found.  The induction ultimately shows that $\Sigma$ must be a ``near-graph component stack."  A global argument then finds a suitable $M$.  Therefore, $\Sigma$ cannot in fact be a minimizer.    

Having eliminated all nonstandard double bubbles from consideration, the only possible minimizer left is the standard double bubble.  

\section{Delaunay hypersurfaces}

Constant-mean-curvature hypersurfaces of revolution are known as Delaunay hypersurfaces \cite{HutchingsMorganRitoreRos00dblbblR3, Hsiang82constantmeancurvature, HsiangYu81delaunayRn, Delaunay1841, Eels87delaunay}.  Let $\Sigma \subset \bf{R^n}$ be a constant-mean-curvature hypersurface invariant under the action of the group $O(n)$ of isometries fixing the axis $L$.  $\Sigma$ is generated by a curve $\Gamma$ in a plane containing $L$.  Put coordinates on the plane so $L$ is the $x$ axis.  Parameterize $\Gamma = \{x(t),y(t)\}$ by arc-length $t$ and let $\theta(t)$ be the angle from the positive $x$-direction up to the tangent to $\Gamma$.  Then $\Gamma$ is determined by the differential equations
\begin{equation} \label{e:diffeq}
\begin{split}
\dot{x} &= \cos \theta\\
\dot{y} &= \sin \theta\\
\dot{\theta} &= -(n-1) H + (n-2) \frac{\cos\theta}{y}
\end{split}
\end{equation}
Here, $\tan\theta = \dot{y}/\dot{x}$ is the slope of $\Gamma$, $\kappa = -\dot{\theta}$ is the planar curvature of $\Gamma$ with respect to the normal $N = (\sin\theta, -\cos\theta)$, and $H$ is the (constant) mean curvature of $\Sigma$ with respect to $N$ (an average of the planar curvature $\kappa$ and the curvature due to rotating about $L$).  

The {\em force} of $\Gamma$ with respect to $N$ is a constant given by 
\begin{equation} \label{e:force}
F = y^{n-2} (\cos\theta - H y) \enspace .
\end{equation}
Note that from Eqs.~\eqnref{e:diffeq} and~\eqnref{e:force}, 
\begin{equation} \label{e:ddottheta}
\ddot{\theta} = -\frac{(n-1)(n-2)}{y^n} F \sin \theta \enspace .
\end{equation}

\begin{figure}
\includegraphics[scale=.65]{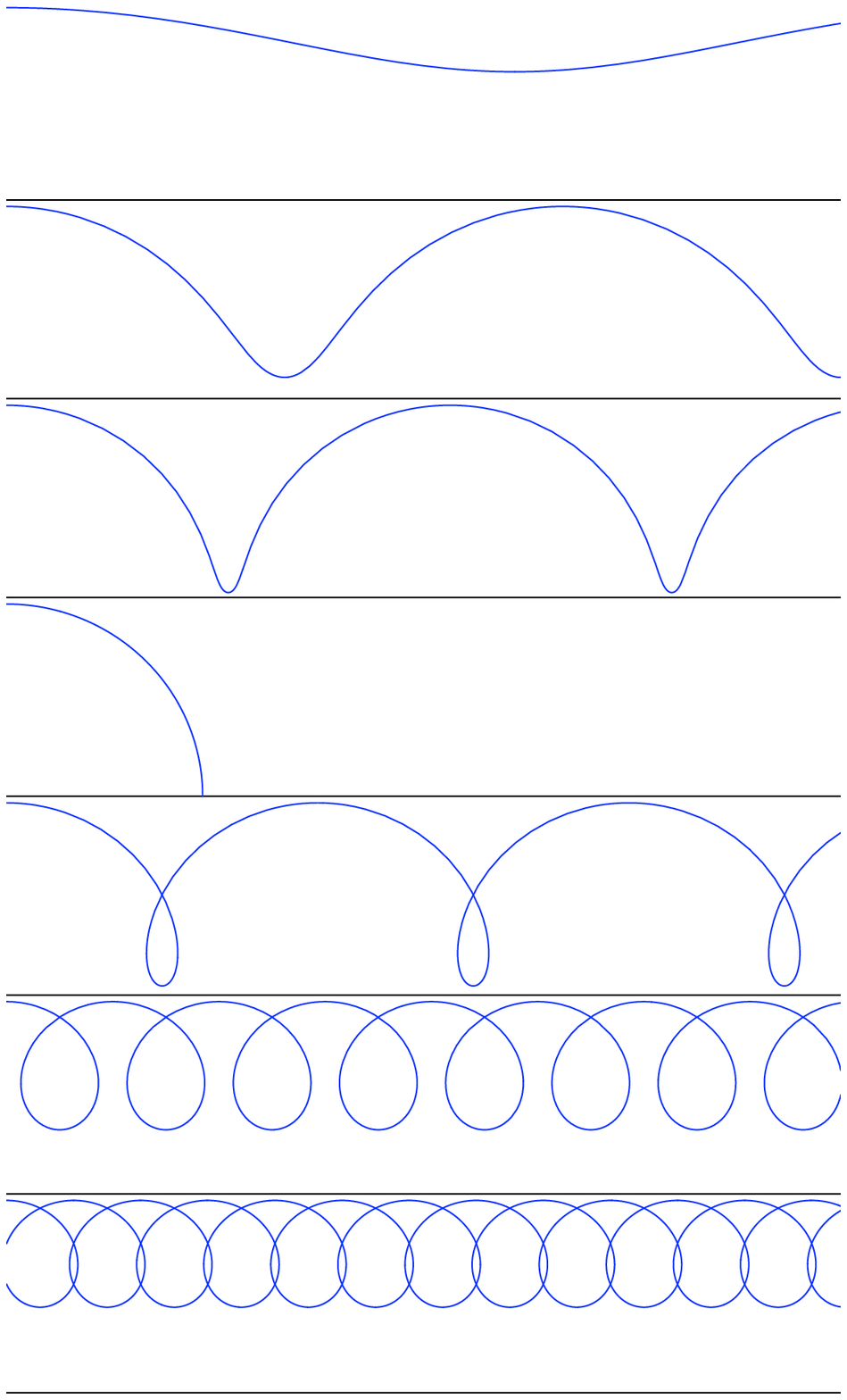}\\
\includegraphics[scale=.65]{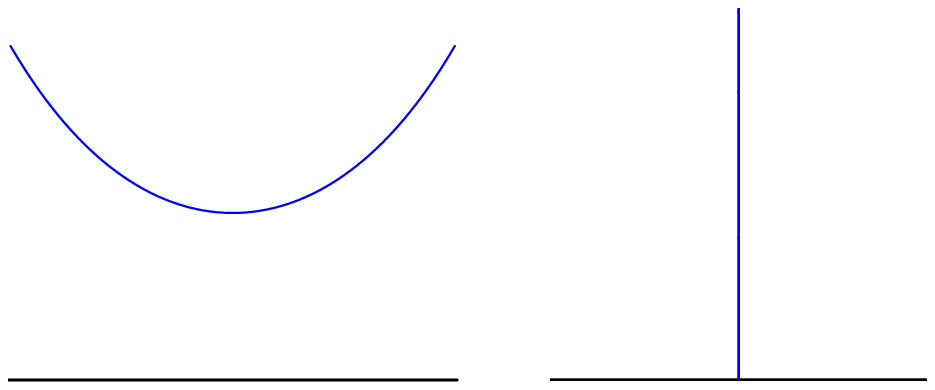}
\caption{Smooth regions of the cluster are parts of ``Delaunay" hypersurfaces of revolution.  Plotted are the generating curves of Delaunay hypersurfaces in $\bf{R^3}$ all starting at $(x,y)=(0,1)$ and $\theta=0$, with increasing mean curvature from top to bottom (positive downward at $x=0$): three unduloids, a sphere, then three nodoids.
Below are examples of a catenoid and a vertical hyperplane, the Delaunay hypersurfaces of zero mean curvature.
} \label{f:delaunay}
\end{figure}

\begin{theorem}[{\cite[Prop.~4.3]{HutchingsMorganRitoreRos00dblbblR3}}] \label{t:delaunay}
Let $\Gamma$ be a complete upper half-planar generating curve that, when rotated about $L$, generates a hypersurface $\Sigma$ with constant mean curvature.  The pair $(H,F)$ determines $\Gamma$ up to horizontal translation.
\begin{enumerate}
\item If $H=0$ and $F\neq 0$, then $\Gamma$ is a curve of catenary type and $\Sigma$ is a hypersurface of catenoid type.  
\item If $H F < 0$, then $\Gamma$ is a locally convex curve and $\Sigma$ is a nodoid.  
\item If $H F > 0$, then $\Gamma$ is a periodic graph over $L$ and $\Sigma$ is an unduloid or a cylinder.  
\item If $H = F = 0$, then $\Gamma$ is a ray orthogonal to $L$ and $\Sigma$ is a vertical hyperplane.
\item If $F=0$ and $H \neq 0$, then $\Gamma$ is a semi-circle and $\Sigma$ is a sphere.  
\end{enumerate}
\end{theorem}

See \figref{f:delaunay}.  The Delaunay hypersurfaces with nonzero mean curvature are the sphere, unduloid and nodoid.  If $\Sigma$ has positive mean curvature upward then it must be a nodoid.  If $\Gamma$ is not graph, then $\Sigma$ must be either a nodoid or a hyperplane.  

\begin{lemma}[Force balancing~{\cite[Lemma~4.5]{HutchingsMorganRitoreRos00dblbblR3}}] \label{t:forcebalancing}
Assume that three generating curves $\Gamma_i$, $i=0,1,2$, of Delaunay hypersurfaces meet at a point $p$.  Consider normals turning clockwise about $p$.  If the curvatures with respect to these normals satisfy $H_0 + H_1 + H_2 = 0$, then the forces with respect to these normals satisfy
\begin{equation}
F_0 + F_1 + F_2 = 0 \enspace .
\end{equation}
\end{lemma}

\noindent The lemma follows from Eq.~\eqnref{e:force}.  

\section{Structure of minimal double bubbles} \label{s:minimal}

A double bubble is a piecewise-smooth oriented hypersurface $\Sigma \subset \bf{R^{n}}$ consisting of three compact pieces $\Sigma_{1}$, $\Sigma_{2}$ and $\Sigma_{0}$ (smooth on their interiors), with a common boundary such that $\Sigma_{1} \cup \Sigma_{0}$, $\Sigma_{2} \cup \Sigma_{0}$ enclose two regions $R_1$, $R_2$, respectively, of given volumes.  The work of Almgren~\cite{Almgren76regularity} (see~\cite[Ch.~13]{Morgan00geometry}) and Hutchings establishes the existence and structure.  

\begin{theorem}[{\cite[Theorem~5.1]{Hutchings97structure}}] \label{t:structure} \label{t:finite}
Any nonstandard minimal double bubble is a hypersurface of revolution about some line $L$, composed of pieces of constant-mean-curvature hypersurfaces meeting in threes at $120$ degree angles.  The bubble is a topological sphere with a finite tree $T$ of annular bands attached, as in \figref{f:structure}.  The two caps of the bottom component are pieces of spheres, and the root of the tree has just one branch.  
\end{theorem}

Hence, any minimal double bubble is determined by an upper half-planar diagram of generating curves that, when rotated about $L$, generate the double bubble.  By studying these generating curves, we will eliminate as unstable nonstandard double bubbles.  

If a double bubble is stable, then for $i = 1, 2$, the mean curvature $H_i$ of $\Sigma_i$ with respect to normals pointing into the region $R_i$ is constant over the various smooth pieces of $\Sigma_i$ and is called the {\em pressure} of $R_i$, positive by \cite[Corollary~3.3]{Hutchings97structure}.  The pressure difference $H_0 = H_1 - H_2$ is the constant mean curvature of $\Sigma_0$ with respect to normals pointing from $R_2$ into $R_1$~\cite[Lemma~3.1]{HutchingsMorganRitoreRos00dblbblR3}.  

Therefore \lemref{t:forcebalancing} applies to each vertex, implying.  

\begin{claim} \label{t:sameouterboundaries}
The outer boundaries of a connected component $C$ -- between $C$ and the outside, ${\bf{R^n}} \setminus (R_1 \cup R_2)$ -- are parts of the same Delaunay hypersurface up to horizontal translation.
\end{claim}

\begin{proof}
The boundary pieces have the same pressure $H$ and -- by applying \lemref{t:forcebalancing} repeatedly -- the same force $F$, so are parts of the same Delaunay hypersurface up to horizontal translation by \thmref{t:delaunay}.  

More carefully, the proof is by induction on the height of the subtree of descendants of $C$ in the associated tree $T$ (from \thmref{t:structure}).  \figref{f:graphstructure} provides examples of the cases that arise.  

\begin{itemize}
\item The base case is when $C$ has no children, so has just a single piece of outer boundary $\Gamma$ (e.g., $\Gamma_5$ in \figref{f:graphstructure}).  Then the claim is trivial.
\item To step the height inductive assumption, assume, e.g., that $\Gamma_3$ and $\Gamma_7$ have the same force.  Force-balancing \lemref{t:forcebalancing} on vertices $v_{123} \equiv \bar{\Gamma}_1 \cap \bar{\Gamma}_2 \cap \bar{\Gamma}_3$ and $v_{278}$ implies that $\Gamma_2$ and $\Gamma_8$ have the same force with respect to normals pointing into $C$.  The same arguments show that $\Gamma_9$ and $\Gamma_{10}$ -- so all outer boundaries of $C$ -- have the same force. \qedhere
\end{itemize}
\end{proof}

Although we will not need it, it is also interesting to remark:
\begin{lemma}[{\cite[Lemma~6.4]{HutchingsMorganRitoreRos00dblbblR3}}] \label{t:smallpressure}
In a minimizing double bubble for unequal volumes, the smaller region has larger pressure.  
\end{lemma}
\noindent \lemref{t:smallpressure} follows from concavity of the minimum-area function~\cite[Theorem~3.2]{Hutchings97structure}.  Note however that in a minimizer for equal volumes, the regions may still have unequal pressures; $H_0$ need not be zero.  

\section{Instability by separation}

Let $\Sigma \subset \bf{R^{n}}$ be a regular stationary double bubble of revolution about axis $L$, with upper half planar generating curves $\cup \Bar{\Gamma}_i$ consisting of arcs $\Bar{\Gamma}_{i}$, with interiors $\Gamma_{i}$, ending either at the axis or in threes at vertices $v_{ijk} = \Bar{\Gamma}_i \cap \Bar{\Gamma}_j \cap \Bar{\Gamma}_k$.  

\def\f {l} %% called f mapping a point to the axis of rotation L

\subsection{\texorpdfstring{The map $\f: \cup \Gamma_{i} \rightarrow L \cup \{\infty\}$}{The map $\f$}}	%% \texorpdfstring is used so the pdf bookmarks are sensible

\begin{definition}
Define a map $\f: \cup \Gamma_{i} \rightarrow L \cup \{\infty\} \equiv [-\infty, +\infty]/(-\infty \sim +\infty)$ that maps $p$ to the intersection with $L$ of the line orthogonal to $\Gamma$ at $p$.  Let $L(p)$ be the ray $\overrightarrow{p \f(p)}$.  Let $\f_\Gamma$ be the restriction of $\f$ to the arc $\Gamma$.  For $p$ an endpoint of $\Gamma$, define $\f_\Gamma(p) \in [-\infty,+\infty]$ to be the limiting value of $\f_\Gamma$ approaching $p$, $\lim_{\substack{q \rightarrow p\\ q \in \Gamma}} \f(q)$.  (For $\Gamma$ a vertical hyperplane, $\f_\Gamma(p) = \infty$.)
\end{definition}

For short, we will write $\f_i$ for $\f_{\Gamma_i}$, and use similar abbreviated notation for other relevant variables $L$, $\theta$, $\kappa$, $N$, $F$.

For $p$ an endpoint of $\Gamma_i$, note that if $\f_i(p) \in \f(\Gamma_{j})$ and $\Gamma_{j}$ is not a circle or hyperplane ($F_j \neq 0$), then for all $q \in \Gamma_{i}$ sufficiently close to $p$, $\f(q) \in \f(\Gamma_{j})$.  (This follows, e.g., because $\dot{\f}_j$ is proportional to the force $F_j$; see \cite[Remark~5.1]{HutchingsMorganRitoreRos00dblbblR3}.)

\subsection{\texorpdfstring{Separating set $\f^{-1}(x)$ instability claims}{Separating set $\f\textasciicircum\{-1\}(x)$ instability claims}}

\begin{theorem}[{\cite[Prop.~5.2]{HutchingsMorganRitoreRos00dblbblR3}}] \label{t:separation}
Consider a stable double bubble of revolution $\Sigma \subset \bf{R^{n}}$, $n \geq 3$, with axis $L$.  Assume that there is a minimal set of points $\{p_{1}, \ldots, p_{k}\}$ in $\cup \Gamma_{i}$ with $\f(p_{1}) = \cdots = \f(p_{k}) = x$ that separates $\cup \Bar{\Gamma}_i$.  

Then every connected component of $\Sigma$ that contains one of the points $p_{i}$ is part of a sphere centered at $x$ (if $x \in L$) or part of a hyperplane orthogonal to $L$ (in the case $x = \infty$).  
\end{theorem}

We sketched the proof of \thmref{t:separation} in our introduction \secref{s:instability}.  We will apply several useful corollaries of this theorem, taken directly from \cite{ReichardtHeilmannLaiSpielman03dblbblR4}.  

\begin{corollary}[{\cite[Cor.~4.2]{ReichardtHeilmannLaiSpielman03dblbblR4}}] \label{t:downvertical}
No generating curve that turns downward past the vertical can have an internal separating set, i.e., two points $p_1 \neq p_2$ in the arc, with $\f(p_1) = \f(p_2)$.    
\end{corollary}
\begin{proof}
Otherwise, by \thmref{t:separation}, the arc would have to be part of either a circle with center on the axis $L$ or a line perpendicular to $L$.  But neither turns past the vertical.  
\end{proof}

\begin{corollary}[{\cite[Cor.~4.3]{ReichardtHeilmannLaiSpielman03dblbblR4}}] \label{t:twicevertical}
No generating curve that is not part of a vertical line can go vertical twice, including at least once in its interior.  
\end{corollary}
\begin{proof}
Such an arc (a nodoid by \thmref{t:delaunay}) has a separating set $\f^{-1}(x)$ for some $x$ with $\abs{x}$ large enough, contradicting \corref{t:downvertical}.  
\end{proof}

\begin{corollary}[{\cite[Cor.~4.4]{ReichardtHeilmannLaiSpielman03dblbblR4}}] \label{t:separatingset}
Consider a nonstandard minimizing double bubble.  Then there is no $x \in L \cup \{\infty\}$ such that $\f^{-1}(x) \smallsetminus (\text{two circular caps})$ contains points in the interiors of distinct $\Gamma_{i}$ that separate $\cup \Bar{\Gamma}_i$.  
\end{corollary}
\begin{proof}[Proof sketch]
For $x \in L$, the statement is \cite[Prop.~5.7]{HutchingsMorganRitoreRos00dblbblR3}.  Arguments using ``force balancing'' show that more pieces of the minimizer are spherical and hence the bubble is the standard double bubble.  For $x = \infty$, note that a separating set crosses at least one outer boundary.  By \thmref{t:separation}, this boundary is a vertical line, contradicting positive pressure of the regions.  
\end{proof}

We will consider various nonstandard double bubbles, and show that they violate one of the above corollaries of \thmref{t:separation}, hence cannot be minimizing.  

These corollaries sufficed for the proof of the Double Bubble Conjecture in $\bf{R^4}$, but in higher dimensions $\bf{R^n}$ we will need to use more information about Delaunay hypersurfaces, Lemmas~\ref{t:undularyray} and~\ref{t:ddottheta}.

\section{Delaunay hypersurface lemmas} \label{s:delaunaylemmas}

We will need a few more properties of Delaunay hypersurfaces, which can be proven by comparing $\Sigma$ to certain spheres.

\begin{lemma} \label{t:undularyray}
Consider an unduloid generated by $\Gamma$ rising to a point $p$ where it makes an angle $\theta_\Gamma(p) \in [\pi/2-2\psi,\pi/2)$ above the horizontal.  Then the ray going back down from $p$ at an angle of $\psi \in (0,\pi/4)$ with the unduloid passes completely beneath it to the left of $p$.
\end{lemma}

\begin{proof}
For $\theta_\Gamma(p) \geq \pi/2-\psi$, the ray leaves $p$ heading to the right, so cannot cross the unduloid to the left.  Assume therefore that $\theta_\Gamma(p) \in [\pi/2-2\psi,\pi/2-\psi]$.  (In applications, we will set $\psi=\pi/6$.)  See \figref{f:undularyray}.

\begin{figure}
\includegraphics[scale=0.25]{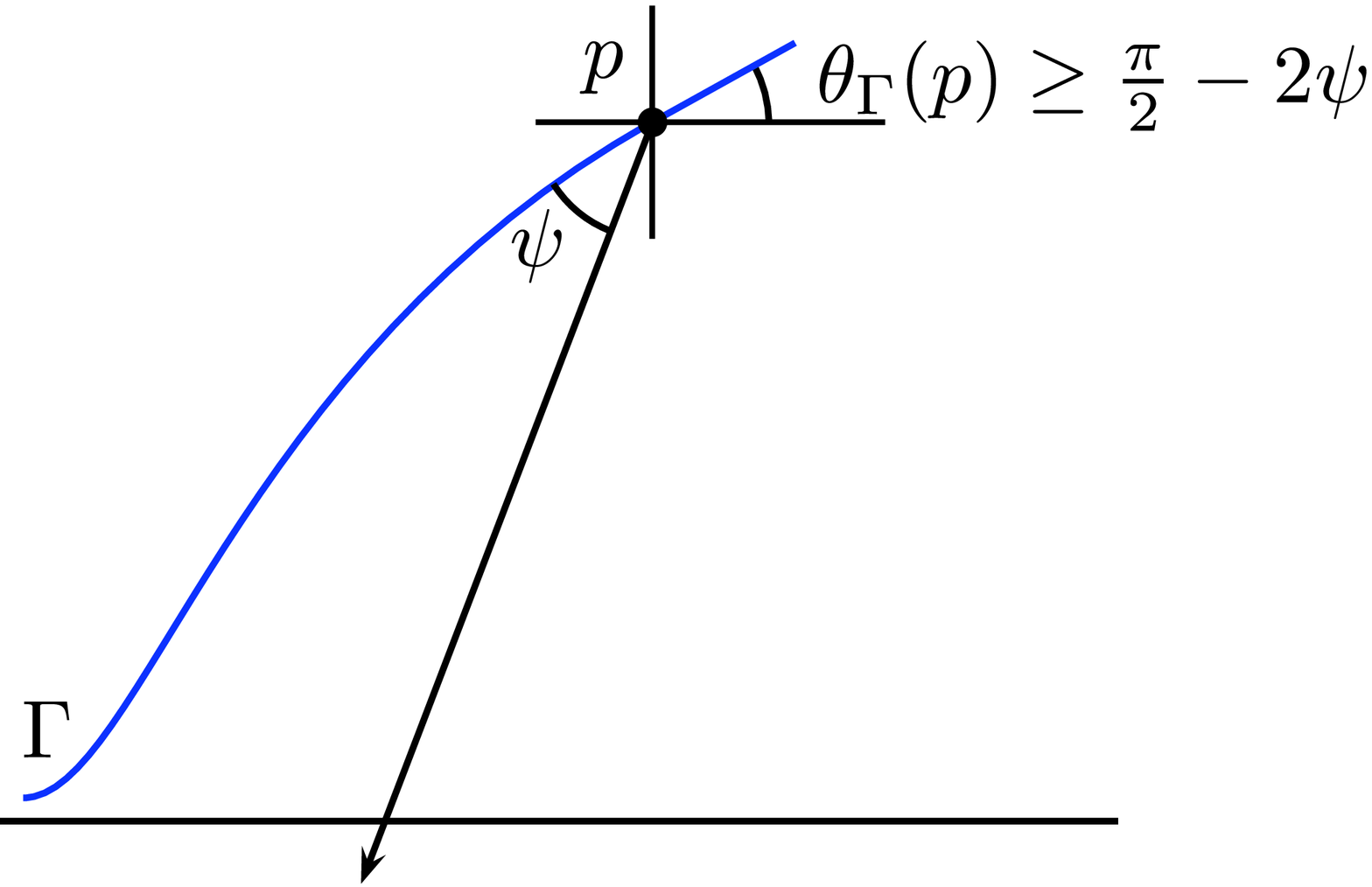}$\qquad\;$
\caption{\lemref{t:undularyray}.} \label{f:undularyray}
\end{figure}

Consider the circle $C$ centered on the $x$ axis and passing through $p$ tangent to the unduloid.  Note that the ray in question passes entirely beneath the semicircle above the $x$ axis; \figref{f:raycircle}.  We will show that the unduloid stays above the semicircle to the left of $p$.  

\begin{figure}
\includegraphics[scale=0.25]{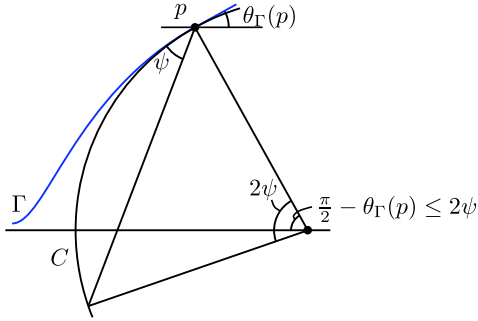}
\caption{The ray from $p$ stays within the upper half of the circle $C$ centered on the $x$ axis and tangent to $\Gamma$ at $p$.  By comparing curvatures, $C$ stays below $\Gamma$ to the left of $p$.} \label{f:raycircle}
\end{figure}

Recall Eq.~\eqnref{e:force} for the force $F = y^{n-2} (\cos\theta - H y)$ of a constant-mean-curvature surface of revolution with respect to a downward-pointing normal.  For an unduloid $F > 0$, while for a sphere $F=0$ (\thmref{t:delaunay}).  Since at $p$, $y$ and $\theta$ are the same, the sphere has higher mean-curvature $H$.  

In particular, $\kappa_\Gamma(p) = -\dot{\theta}_\Gamma(p)$ the planar curvature of $\Gamma$ at $p$ is less than $\kappa_C(p)$ (because the portion of $H$ due to rotating about $L$ is the same for each), so $\Gamma$ leaves $p$ above $C$.  Assume for contradiction that $\Gamma$ crosses $C$ for the first time at a point $q$ to the left of $p$.  Then the unduloid is steeper than the circle at $q$: $\theta_\Gamma(q) > \theta_C(q)$, so $C$ achieves the angle $\theta_\Gamma(q)$ at some point below $q$.  

However, this is impossible; for any angle $\phi$, the $y$-coordinate at which $\Gamma$ first reaches $\phi$, $y(p) - \Delta y_\Gamma(\phi)$, is no more than the $y$-coordinate at which $C$ reaches $\phi$, $y(p) - \Delta y_C(\phi)$.  Indeed, $\ddot{\theta}_\Gamma < 0$ (Eq.~\eqnref{e:ddottheta}), so $\kappa_\Gamma$ is decreasing moving left, while $\kappa_C$ is constant.  $\Gamma$ therefore reaches some maximum angle $\phi_{\text{max}} \geq \theta_\Gamma(p)$.  Beyond $\phi_{\text{max}}$, $\Gamma$ turns upward away from the circle, so the curves can only further diverge.  But for $\phi \in [\theta_\Gamma(p), \phi_\text{max}]$, the difference in $y$ coordinates from $p$ to where the curve reaches $\phi$ is 
\begin{eqnarray*}
\Delta y &=& \int_{\theta_\Gamma(p)}^{\phi} d\psi \frac{dt}{d\psi} \sin \psi \\
&=& \int_{\theta_\Gamma(p)}^{\phi} d\psi \frac{1}{\kappa(\psi)} \sin \psi 
\end{eqnarray*}
Since $\kappa_\Gamma(\psi)$ is decreasing, while $\kappa_C$ is constant, indeed $\Delta y_\Gamma \geq \Delta y_C$.
\end{proof}

\begin{lemma} \label{t:ddottheta}
Consider a smooth curve $\Gamma$ of length $\Delta t$ between points $p$ and $q$.  Let $\theta(t)$ be the angle the curve makes with the horizontal at arc-length $t$ from $p$, with $\theta_p=\theta(0)$ and $\theta_q=\theta(\Delta t)$ the initial and final angles.  Assume $\theta_p \leq \theta_q \leq \theta_p + \pi/3$.  (If we rotate the coordinate system so $\theta_p=0$, then $0 \leq \theta_q \leq \pi/3$.)  Assume $\ddot{\theta} < 0$ and $\theta(t)$ in $[\theta_p, \theta_q]$.  Then the line through $q$ at angle $\pi/6$ clockwise from $\Gamma$ (i.e., at angle $\theta_q - \pi/6$ above horizontal) passes above $\Gamma$.
\end{lemma}

\begin{proof}
Rotate the coordinate system so without loss of generality $\theta_p = 0$ and $\theta_q \leq \pi/3$.  Then $\ddot{\theta}<0$ implies that for all $t \in [0,\Delta t]$,
$$
\theta(t) > \tfrac{t}{\Delta t} \theta_q \enspace .
$$
Let
\begin{eqnarray*}
\Delta y &=& y(\Delta t) - y(0) \\
&=& \int_0^{\Delta t} \sin(\theta(t)) dt \\
&>& \int_0^{\Delta t} \sin(\tfrac{t}{\Delta t}\theta_q) dt\\
&=& \Delta t (1 - \cos\theta_q)/\theta_q \enspace ,
\end{eqnarray*}
where for the inequality we used that $\sin\phi$ is increasing for $\phi \in (-\pi/2,\pi/2)$.
Similarly, let
\begin{eqnarray*}
\Delta x &=& x(\Delta t) - x(0) \\
&=& \int_0^{\Delta t} \cos(\theta(t)) dt \\
&<& \int_0^{\Delta t} \cos(\tfrac{t}{\Delta t}\theta_q) dt \\
&=& \Delta t (\sin\theta_q)/\theta_q \enspace ,
\end{eqnarray*}
using that $\cos\phi$ is decreasing for $\phi \in (0,\pi)$.
Hence,
\begin{eqnarray*}
\frac{\Delta y}{\Delta x} &>& \frac{1 - \cos(\theta_q)}{\sin(\theta_q)} \\
&\geq& \tan(\theta_q - \pi/6) \enspace .
\end{eqnarray*}
$\Gamma$ is a convex-up curve between $p$ and $q$, so it lies beneath the direct line segment $pq$, which has higher slope than, and hence lies beneath, the line through $q$ at angle $\theta_q - \pi/6$ above horizontal.
\end{proof}

By \thmref{t:delaunay} and Eq.~\eqnref{e:ddottheta}, on a clockwise-turning nodoid, $\ddot{\theta} < 0$ provided $\dot{y} = \sin\theta < 0$; while for an unduloid with $t$ increasing to the right, $\ddot{\theta} < 0$ provided $\dot{y} > 0$.

\begin{corollary} \label{t:unduloidcorollary}
Consider an unduloid or nodoid generating curve strictly increasing left to right from $p$ to $q$ with the notation of \figref{f:boundary}.  Assume $\theta_{3}(p) \in [0, \pi/2)$ and
$$
\max\{\pi/6, \theta_{3}(p)\} < \theta_{3}(q) \leq \min\{\pi/2, \theta_{3}(p) + \pi/3\}\enspace .$$
Then $\f_{4}(q) < \f_{2}(p)$.
\end{corollary}

\begin{figure} 
\includegraphics[scale=.3]{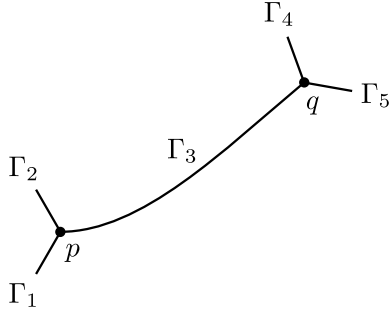}
\caption{Notation for a boundary involving five arcs, $\Gamma_1$ through $\Gamma_5$, for Corollaries~\ref{t:unduloidcorollary} and~\ref{t:graphcorollary}.} \label{f:boundary}
\end{figure}

\begin{proof}
By \lemref{t:ddottheta}, the ray $L_4(q) = \overrightarrow{q \f_4(q)}$ (the extension of the normal $N_4(q)$) passes above $p$ and by assumption it is clockwise from $N_2(p)$, so it stays above $L_2(p) = \overrightarrow{p \f_2(p)}$.  
\end{proof}

\begin{corollary} \label{t:graphcorollary}
Consider a piece of boundary, arriving at $q$ from the left with the notation of \figref{f:boundary}.  Assume $\theta_{3}(q) \in [\pi/6,\pi/2)$ and $\theta_{3}(p) \in [\pi/6,\theta_{3}(q)+\pi/3]$.  Also assume that $\theta_{3} \neq \pi \mod 2\pi$ between $p$ and $q$.
Then $\f_{5}(q) > \f_{1}(p)$.
\end{corollary}

\begin{proof}
There are three cases, depending on whether $\Gamma_3$ is an unduloid, a convex-up catenoid or nodoid, or a nodoid convex-down at $q$.  The assumption that $\theta_{\Gamma_3} \neq \pi \mod 2\pi$ between $p$ and $q$ prevents a nodoid from turning too far around between $p$ and $q$.
\begin{enumerate}
\item If $\Gamma_3$ is an unduloid, then by \lemref{t:undularyray}, 
the ray $L_5(q)$ passes downward beneath $\Gamma_3$, so in particular stays right of $p$.  (This remains true if the unduloid passes through one or more minima between $p$ and $q$.)  
\item If $\Gamma_3$ is a convex-up catenoid or nodoid, then convexity implies again that $L_5(q)$ stays right of $\Gamma_3$.
\item If $\Gamma_3$ is a convex-down nodoid, then by \lemref{t:ddottheta} with $\psi = \pi/6$ again $L_5(q)$ stays right of $\Gamma_3$.  (Applying the lemma requires reflecting horizontally and using that $\Gamma_3$ turns no more than $\pi/3$ radians from $p$ to $q$.)
\end{enumerate}
In all cases, since $\Gamma_1$ leaves $p$ clockwise of how $\Gamma_5$ leaves $q$, $\f_{1}(p) < \f_{5}(q)$.
\end{proof}

\section{Rotation notation}  \label{s:notation}

A nonstandard minimal double bubble's generating curves can be classified by the angles at which arcs leave each vertex.  For our analysis, it will suffice to know which arcs are leaving from the right and which from the left; there are therefore six cases, each covering $\pi/3$ radians, which we term ``notches."   Incrementing the rotation notch of a vertex corresponds to turning it counterclockwise until an arc passes the vertical, as occurs for $\Gamma_{3}$ from \figref{f:graphstructure} to \figref{f:offgraphstructure}, and from \figref{f:rotation}(a) to \ref{f:rotation}(b).  The extreme position with an arc leaving a vertex exactly at the vertical divides two consecutive notch cases.  If the limiting value of $\f$ along the vertical arc is $+\infty$ (or if the arc is a vertical line), the position is assigned the smaller rotation notch value, and if the limiting value is $-\infty$, the position is given the larger notch value.

\begin{figure}
\includegraphics[scale=.28]{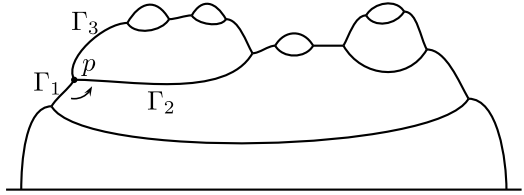}
\caption{From the curves of \figref{f:graphstructure}, vertex  $p = \bar{\Gamma}_1 \cap \bar{\Gamma}_2 \cap \bar{\Gamma}_3$ has turned one counterclockwise ``notch,'' since $\Gamma_{3}$ has passed the vertical.} \label{f:offgraphstructure}
\end{figure}

\begin{figure} 
\includegraphics[scale=.28]{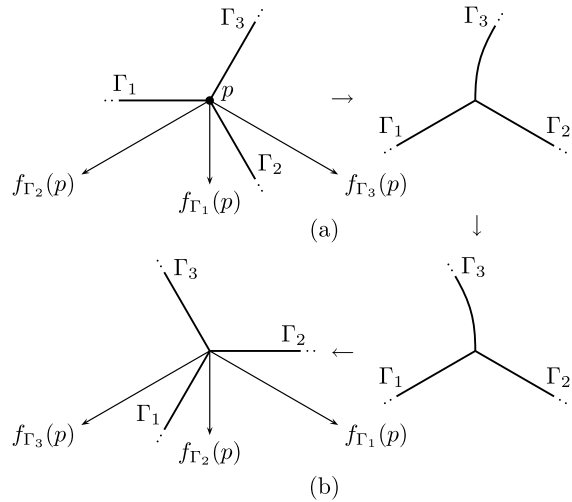}
\caption{A close-up view of $p$ as it turns one notch counterclockwise.  In (a), $\f_2(p) < \f_1(p) < \f_3(p)$; on the right, $\f_3(p) = +\infty$.  In (b), $\f_3(p) < \f_2(p) < \f_1(p)$; $\f_3(p) = -\infty$ on the right.} \label{f:rotation}
\end{figure}

The rotation numbers for our earlier $4 + 4$ bubble example are indicated in \figref{f:rotationexample}.  

\begin{figure} 
\includegraphics[scale=.28]{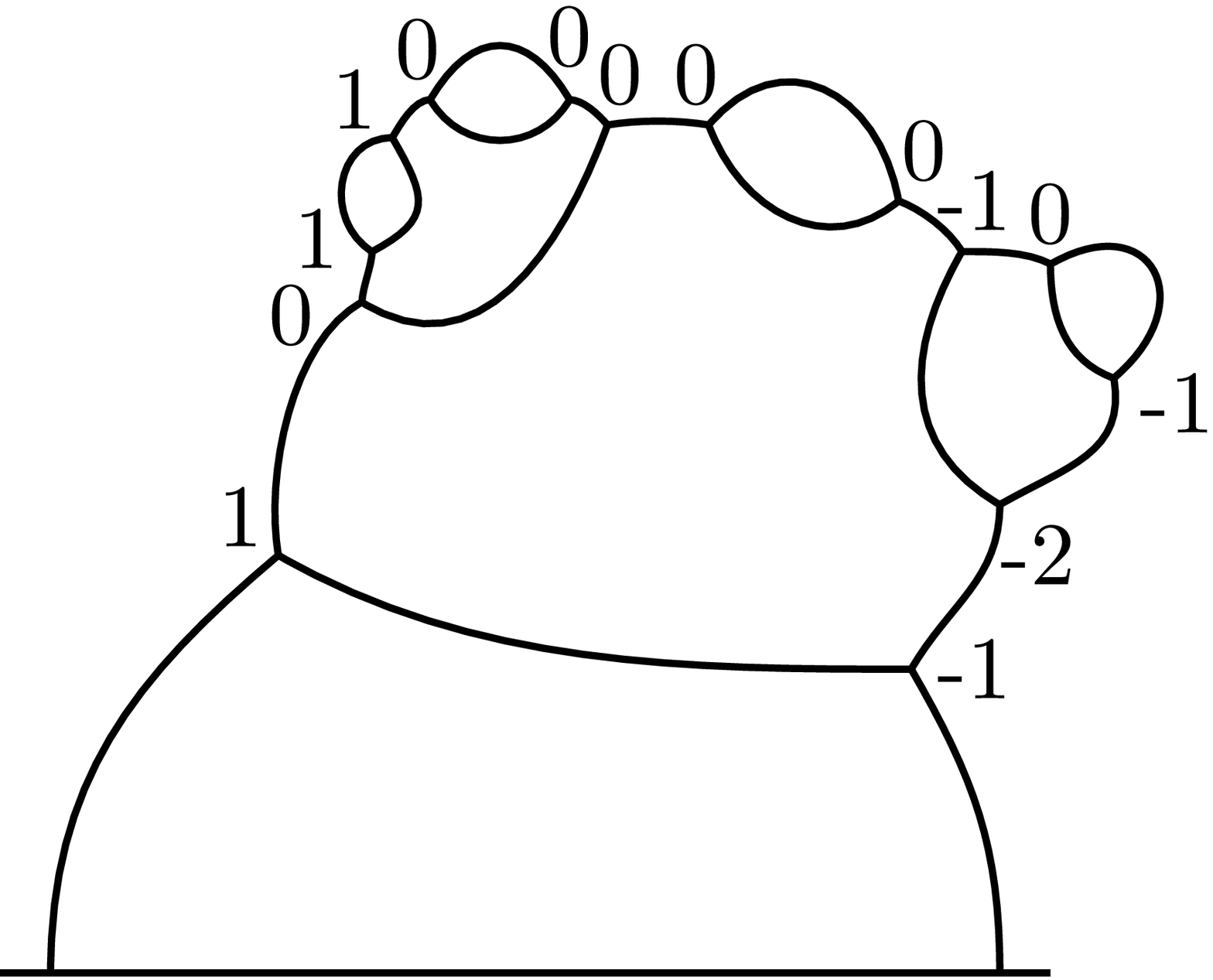}
\caption{This nonstandard double bubble has the same associated tree structure as those in Figures~\ref{f:graphstructure} and \ref{f:offgraphstructure}, but some vertices have been rotated a notch or two counterclockwise or clockwise, as marked.} \label{f:rotationexample}
\end{figure}

\section{Near-graph leaves and bubble stacks} \label{s:neargraph}

\subsection{Near-graph leaves} \label{s:neargraphleaf}

A ``leaf'' of a nonstandard minimizing double bubble corresponds to a leaf of its associated tree of \thmref{t:structure} and \figref{f:structure}.  We will use the notation of \figref{f:zerozeroleaf}. 

\begin{definition}
A leaf is {\em near graph} if vertices $p$ and $q$ are rotated respectively $(0,0)$ notches, $(0,1)$ notches or $(-1,0)$ notches from their positions in \figref{f:zerozeroleaf}.  (\figref{f:zerooneleaf} gives examples for the latter two cases.)  

We say that a near-graph leaf is {\em right-side up} if, with the notation of \figref{f:zerozeroleaf}, the region below $\Gamma_2$ is inside the bubble; and {\em upside down} if the region above $\Gamma_3$ is inside the bubble.
\end{definition}

\begin{figure}
\includegraphics[scale=.28]{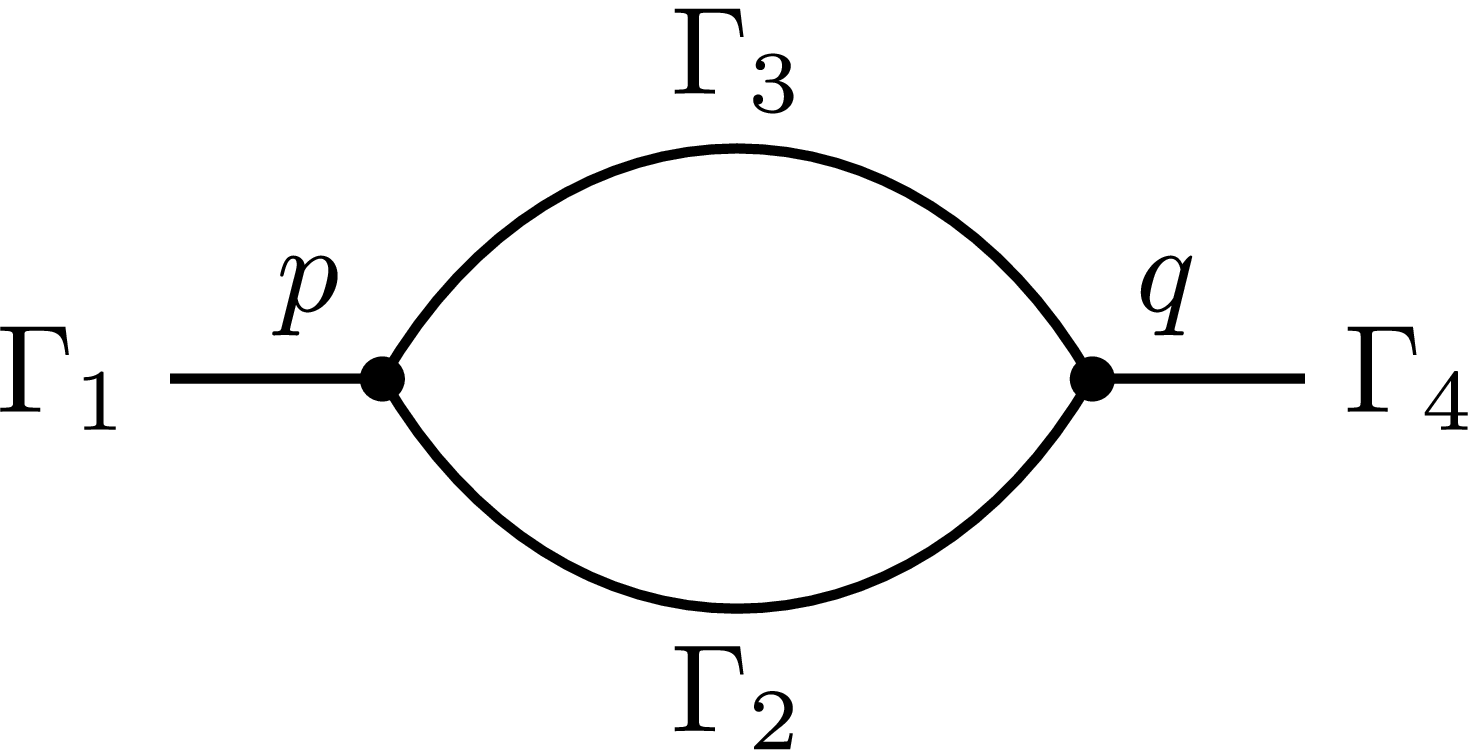}
\caption{A leaf involves four arcs: $\Gamma_{1}$, $\Gamma_{2}$, $\Gamma_{3}$, $\Gamma_{4}$.  In general each vertex can be rotated $m$ notches counterclockwise from the pictured configuration, in which all arcs are graphs and $m_p = m_q = 0$.} \label{f:zerozeroleaf}
\end{figure}

\begin{figure}
\includegraphics[scale=.28]{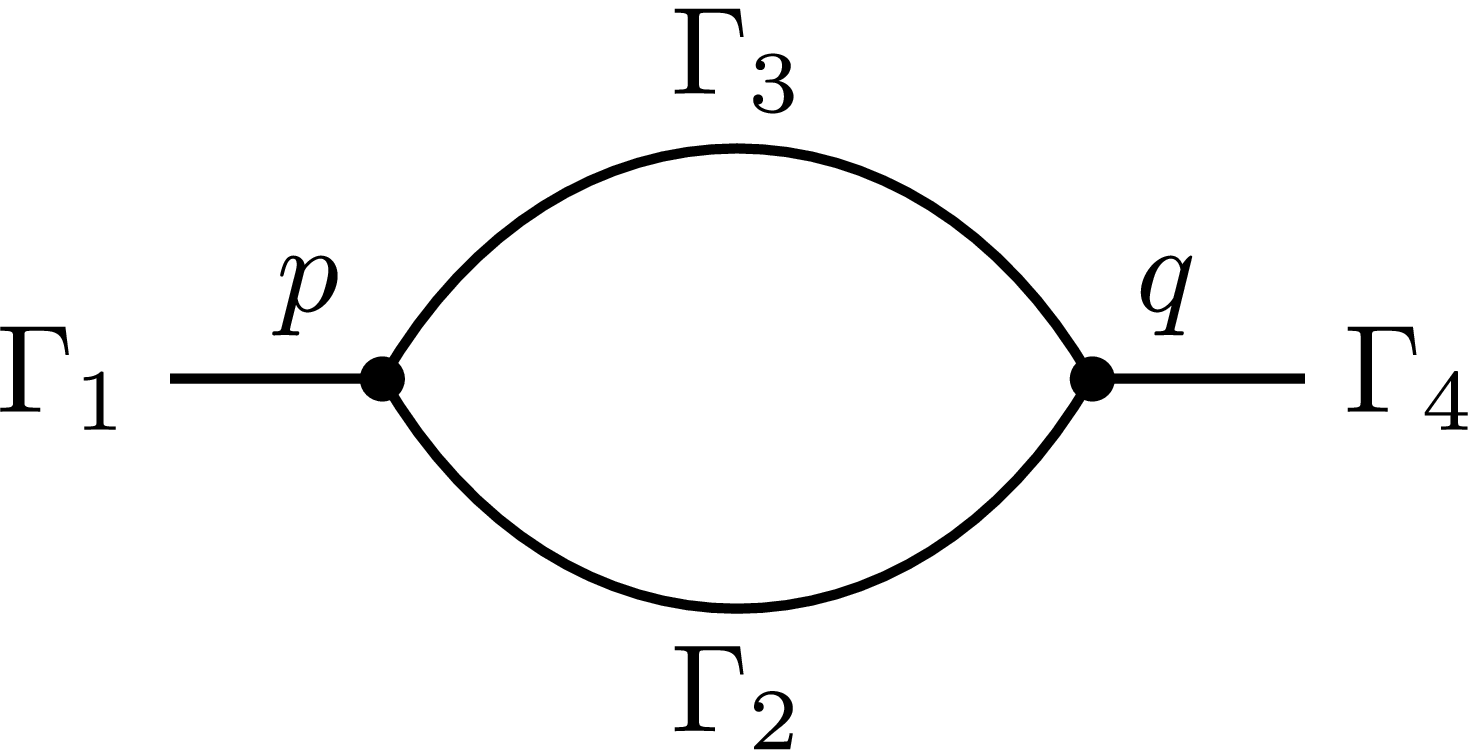}$\qquad$
\includegraphics[scale=.28]{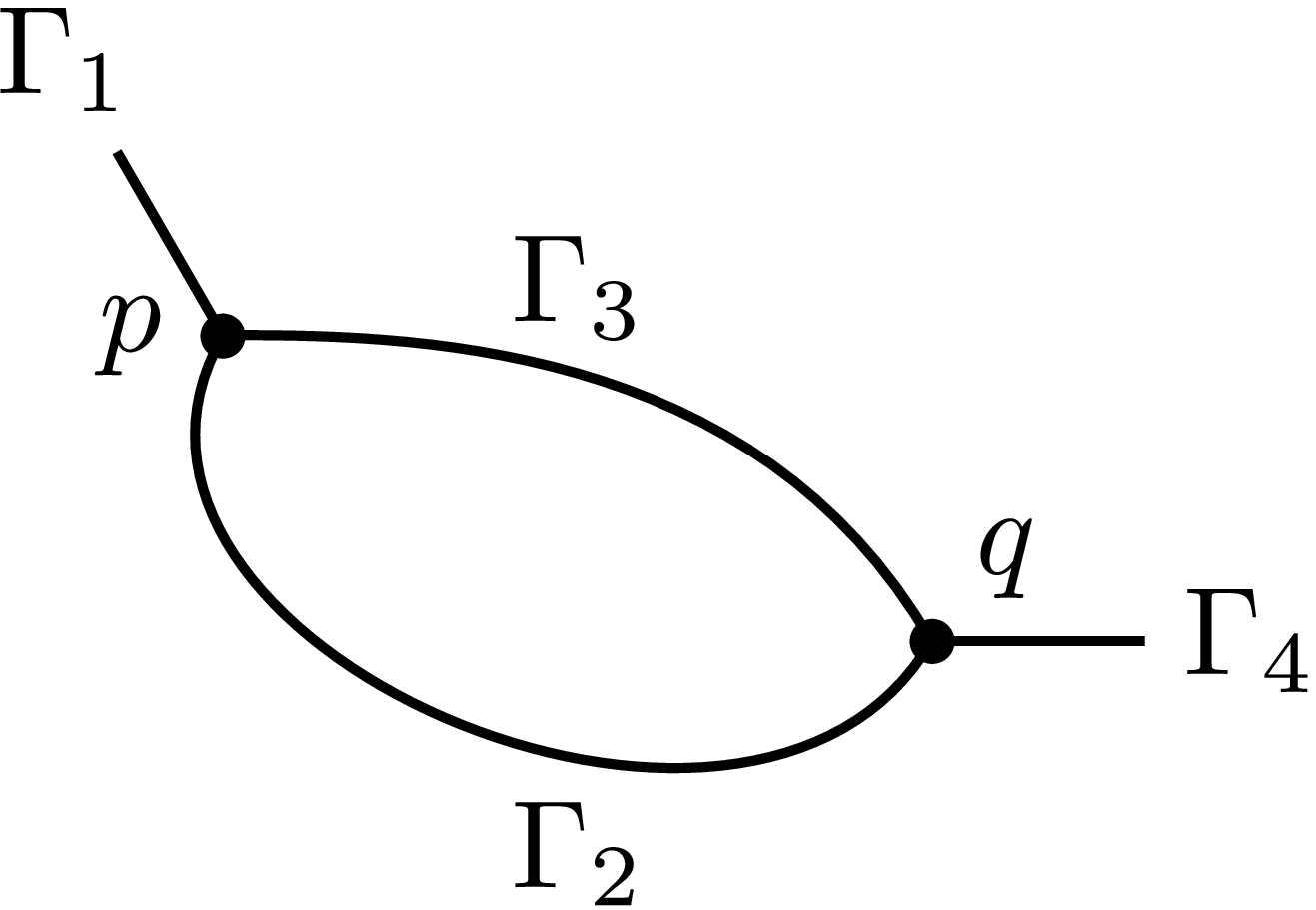}
\caption{Leaves with $p$ and $q$ vertices rotated respectively $(0,1)$ notches counterclockwise from their positions in \figref{f:zerozeroleaf} (left), or $(-1,0)$ notches (right).} \label{f:zerooneleaf}
\end{figure}

\begin{claim} \label{t:neargraphleafhigherpressure}
For an upside-down, near-graph leaf component, the region containing the leaf must have strictly higher pressure.  
\end{claim}
\begin{proof}
If $\Gamma_3$ has positive pressure up in Figures~\ref{f:zerozeroleaf} or~\ref{f:zerooneleaf}, then $\Gamma_3$ must be a convex-up nodoid.  But then the angle constraints at vertices $p$ and $q$ cannot be satisfied, a contradiction.  If there is zero pressure across $\Gamma_3$, then $\Gamma_3$ can be a catenoid or a vertical hyperplane, but again neither matches the angle constraints.\end{proof}

\subsection{Right-side-up, near-graph component stacks} \label{s:neargraphstack}

\begin{definition}
A component stack consists of a base component and all its descendants in the associated tree $T$.  The base component cannot be the root component of $T$.  A right-side-up, graph component stack is a component stack in which all the associated generating curves are graph, and the root is downward from the base component.  
A right-side-up, near-graph component stack -- or ``near-graph stack," for short -- is the same, except for each internal boundary the left vertex can be rotated one notch clockwise or the right vertex rotated one notch counterclockwise.
\end{definition}

Examples are in Figures~\ref{f:zerozeroleaf} and \ref{f:zerooneleaf} (if the path to the root passes through $\Gamma_2$), and in \figref{f:neargraphcomponentstackexamples}.

\begin{figure}
\includegraphics[scale=.8]{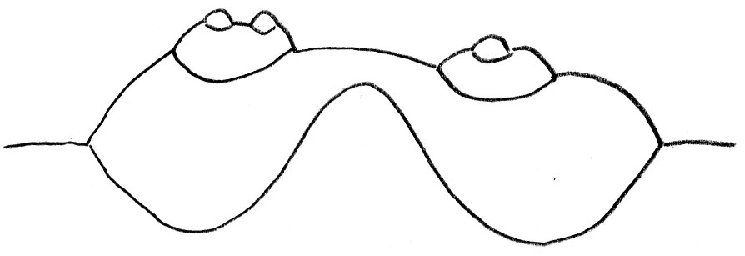}\\
\includegraphics[scale=.8]{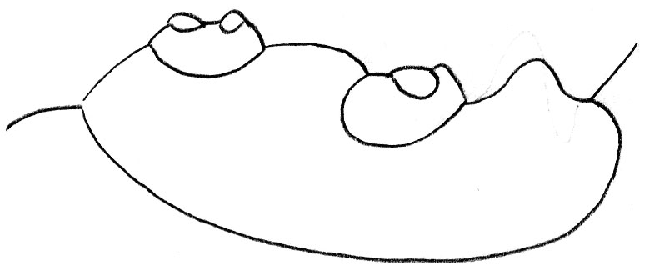}\\
\caption{Two examples of right-side-up, near-graph component stacks.  The first is in fact a graph stack.} \label{f:neargraphcomponentstackexamples}
\end{figure}

\begin{figure}
\includegraphics[scale=.9]{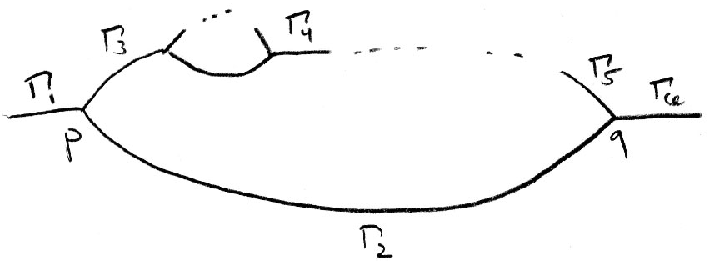}\\
\includegraphics[scale=.9]{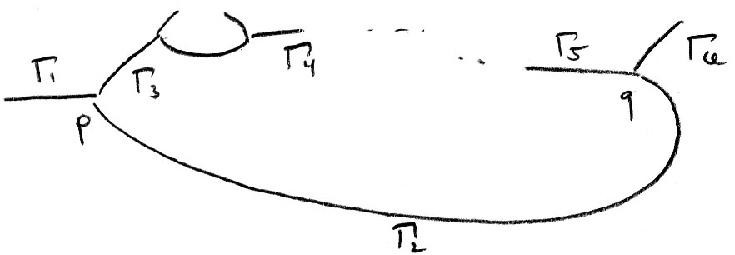}
\caption{Notation for right-side-up, near-graph component stacks.} \label{f:neargraphcomponentstackgeneric}
\end{figure}

\begin{lemma} \label{t:neargraphstack}
For a right-side-up, near-graph component stack in a minimizer, with the notation as in \figref{f:neargraphcomponentstackgeneric}, $\f_1(p) < \f_6(q)$.  
\end{lemma}

\begin{proof}
Assume without loss of generality that vertex $p$ is rotated as in \figref{f:neargraphcomponentstackgeneric} -- if in fact $p$ is rotated $m_p = -1$ notches, then flipping the picture horizontally gives the desired positioning.  

If $m_q = 1$ (second case of \figref{f:neargraphcomponentstackgeneric}), then $\Gamma_2$ is a convex-up nodoid since it turns past the vertical, and the ray $L_6(q)$ therefore stays above $\Gamma_2$, implying $\f_1(p) < \f_6(q)$.  

Therefore assume $p$ and $q$ are rotated as in the first case of \figref{f:neargraphcomponentstackgeneric}; $m_p = m_q = 0$.  
The proof now is by induction on the height of the stack.  

If the component stack is just a leaf (height one), then $\Gamma_3$ goes from $p$ to $q$ as in \figref{f:zerozeroleaf}.  
Both $\f_1(p), \f_4(q) \in \f(\Gamma_2)$.  If $\f_4(q) \leq \f_1(p)$, then $\f_3(p) > \f_1(p)$ and $\f_3(q) < \f_4(q)$ imply that $\f_1(p), \f_4(q) \in \f(\Gamma_3)$, too, giving a $\Gamma_{1,2,3}$ separating set.

If the component stack has height greater than one, then by induction and to prevent internal separating sets $\f(\Gamma_3) < \f(\Gamma_4) < \cdots < \f(\Gamma_5)$.  Since $\f_1(p) < \f_3(p) \leq \sup \f(\Gamma_3)$ and $\inf \f(\Gamma_5) \leq \f_5(q) < \f_6(q)$, again it must that $\f_1(p) < \f_6(q)$.
\end{proof}

\begin{corollary} \label{t:stackdown}
An arc of outer boundary cannot turn downward past the vertical after leaving a right-side-up, near-graph component stack.
\end{corollary}

\begin{proof}
By \lemref{t:neargraphstack}, $\f_1(p) < \f_6(q)$.  If $\Gamma_1$ turns downward past the vertical, then $\f_6(q) \in (\f_1(p), +\infty] \subset \f(\Gamma_1)$, a contradiction of \corref{t:separatingset}.  Similarly $\Gamma_4$ cannot turn downward past the vertical.
\end{proof}

\noindent (\lemref{t:neargraphstack} is a generalization of~\cite[Lemma~7.5]{ReichardtHeilmannLaiSpielman03dblbblR4}, and \corref{t:stackdown} compares to~\cite[Prop.~7.6]{ReichardtHeilmannLaiSpielman03dblbblR4}.)

\section{Leaf and component classification} \label{s:leafcomponentclassification}

We now classify the types of leaves and general components that may occur in a minimizer.  In flavor, this section is similar to the leaf classification~\cite[Prop.~7.1]{ReichardtHeilmannLaiSpielman03dblbblR4} -- however, we will complicate things by considering component stacks and not just leaves; and simplify things by applying Corollaries~\ref{t:unduloidcorollary} and~\ref{t:graphcorollary}.

\begin{theorem} \label{t:componentclassification}
In a minimizer, every leaf must be near graph, and every component stack of height $> 1$ must be near graph and right-side up.
\end{theorem}

\begin{proof}
The proof will be by induction on the height of the stack.  
To reduce repetitive arguments, we will simultaneously consider leaves and component stacks; on a first reading, however, it may be easier to specialize just to leaves.
Diagrams will be condensed by using \raisebox{-.2em}{\includegraphics[scale=.28]{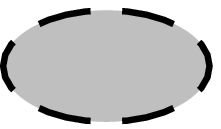}} to indicate the possible presence of one or more near-graph leaves or right-side-up, near-graph stacks.  

Notation: For leaves we will use use the notation of \figref{f:zerozeroleaf}.  Cases will be separated according to vertex rotations in notches away from their positions in that diagram.  For the base component of a component stack of height $>1$, cases will be separated according to vertices' notches away from their positions in one of the cases of \figref{f:stacknotation}.  Here, either $\Gamma_2$ or $\Gamma_3$ is no longer a single arc, but the union of multiple outer boundary arcs (all pieces of the same Delaunay hypersurface by \claimref{t:sameouterboundaries}).  

\begin{figure} 
\includegraphics[scale=.28]{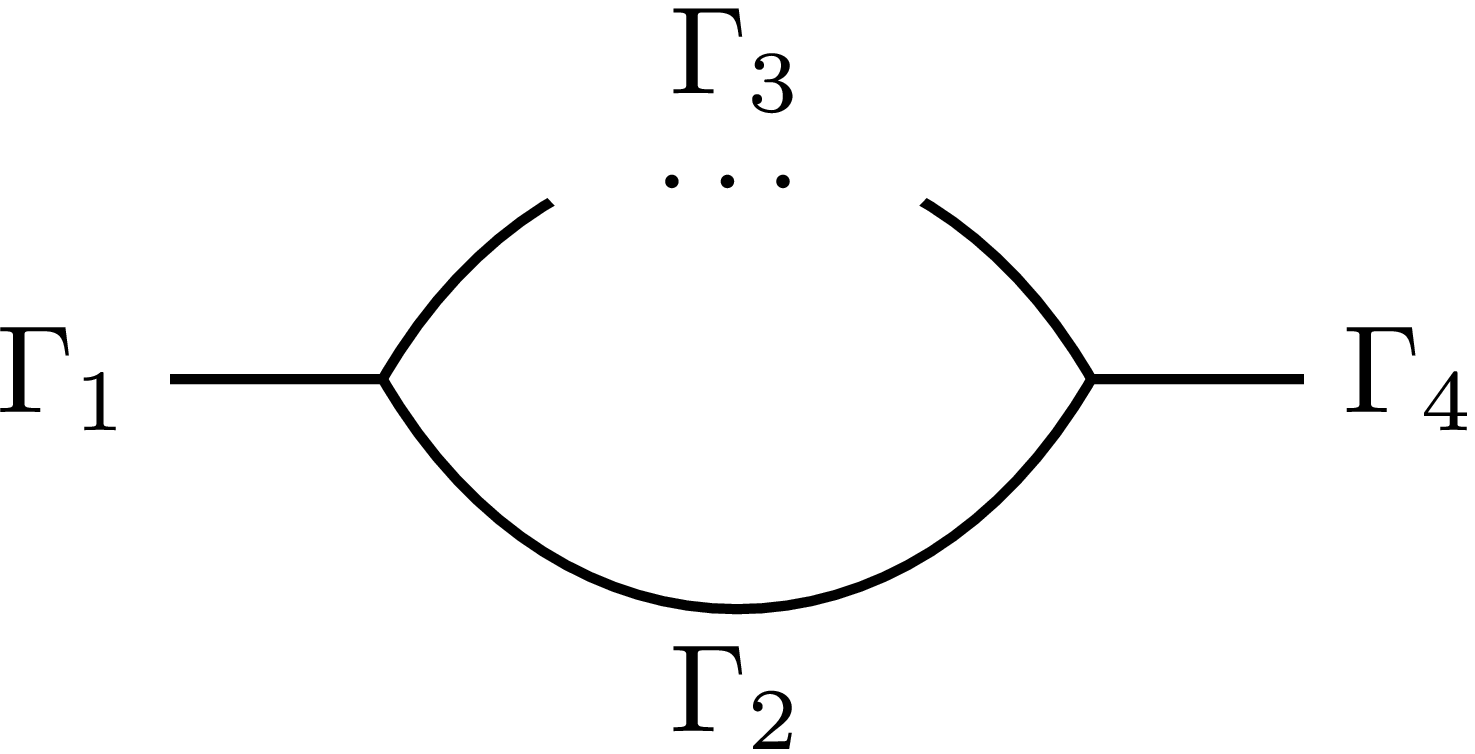}$\qquad$
\includegraphics[scale=.28]{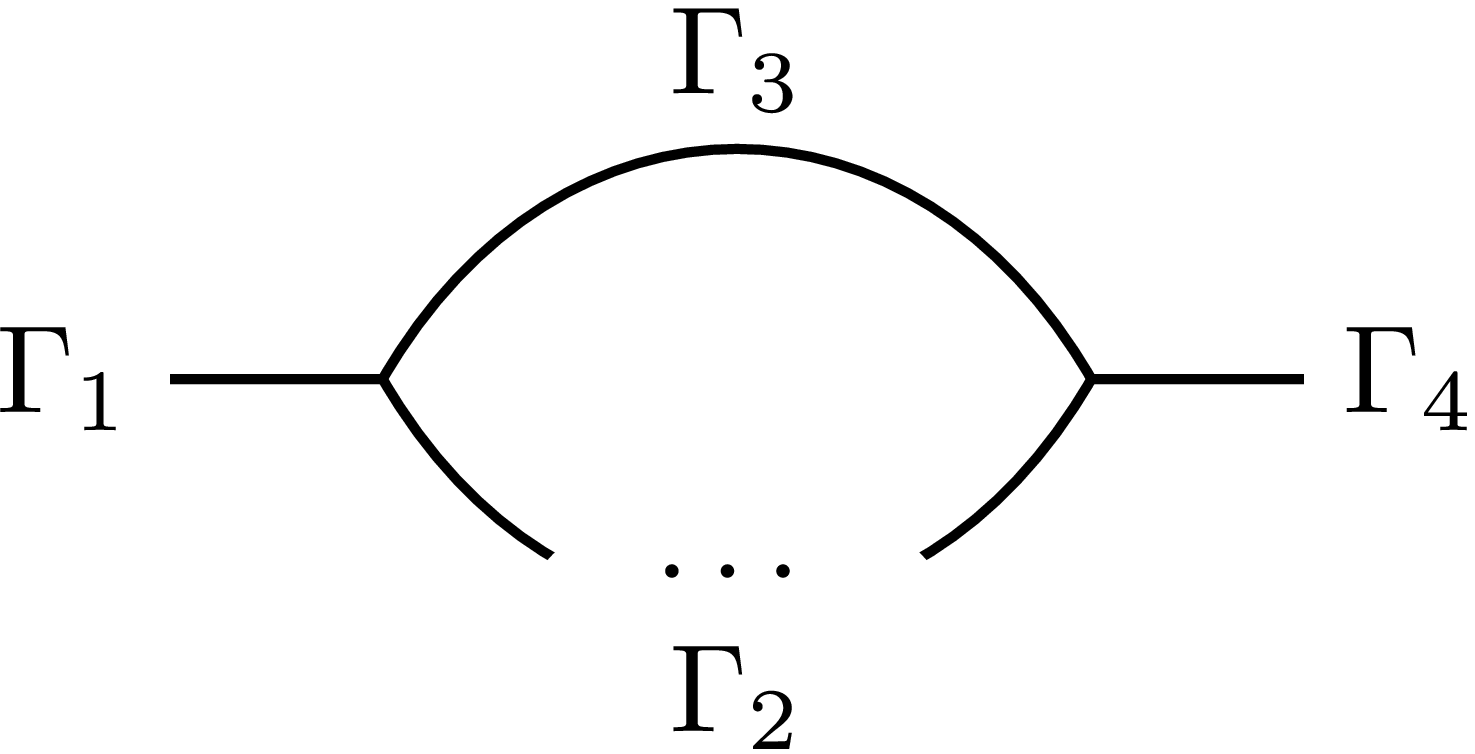}
\caption{Notation for the base component of a stack of bubbles.  On the left (right, respectively), the root of the associated tree is downward (upward, resp.), and leaves are upward (downward).} \label{f:stacknotation}
\end{figure}

Note also the symmetries; case $(i,j)$ is symmetrical under relabeling to $(j-3, i-3)$, and under horizontal reflection to $(-j, -i)$.

The most difficult cases will be $(0,2)$ (symmetrical to $(-1,-3)$) and $(-1,-2)$.

\smallskip\noindent$\bullet$ Cases $(0,0)$ and $(0,1)$: See Figures~\ref{f:(0,0)stack} and~\ref{f:(0,1)stack}.  This case includes all leaves or base components with $\theta_{1}(p) \in [-\pi/6,\pi/6]$ and $\theta_{4}(q) \in [-\pi/6,\pi/2)$ (and $\theta_4(q) = \pi/2$ if $\f_4(q) = +\infty$).  A $(0,0)$ or $(0,1)$ leaf is near graph, so is allowed in the statement of the theorem.  If instead the component is the base of a component stack with one or more near-graph stacks sitting above it, overall we still have an allowed near-graph stack.  If it has a near-graph leaf sitting below it, then by \claimref{t:neargraphleafhigherpressure} the leaf's region has higher pressure.  However, $\Gamma_3$ must have positive pressure down to match the angle constraints at $p$ and $q$, a contradiction.

\begin{figure} 
\includegraphics[scale=.28]{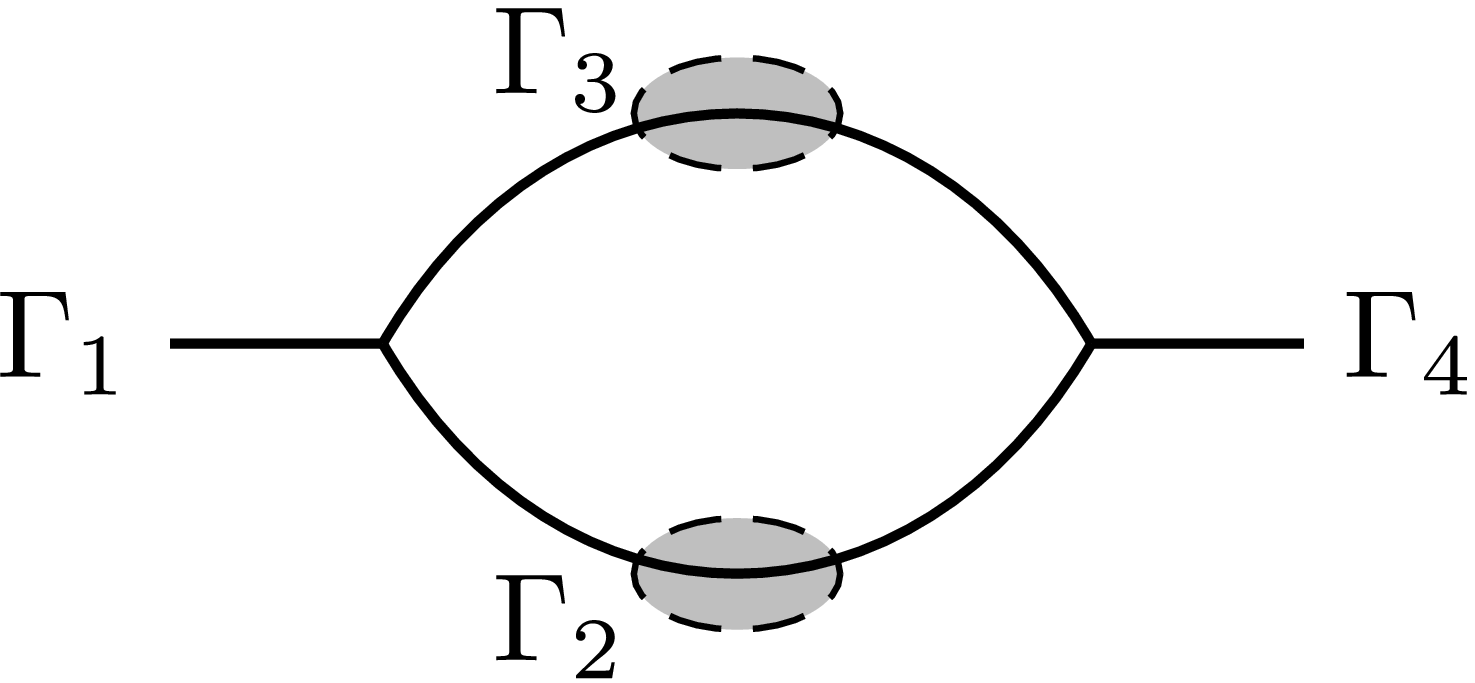}
\caption{Case (0,0).  Either a $(0,0)$ leaf -- i.e., in which vertices $p$ and $q$ are each rotated zero notches from their positions in \figref{f:zerozeroleaf} -- or the $(0,0)$ base component of a stack.  In the latter case, the upper \protect\raisebox{-.2em}{\protect\includegraphics[scale=.28]{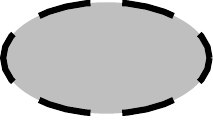}} indicates the possible presence of one or more right-side-up, near-graph component stacks along $\Gamma_3$, while the lower \protect\raisebox{-.2em}{\protect\includegraphics[scale=.28]{images/bubblestack}} indicates the possible presence of upside-down, near-graph leaves along $\Gamma_2$.} 
\label{f:(0,0)stack}
\end{figure}

\begin{figure}
\begin{tabular}{c@{$\quad$}c@{$\quad$}c}
\subfigure[Case (0,1)]{\label{f:(0,1)stack}\includegraphics[scale=.28]{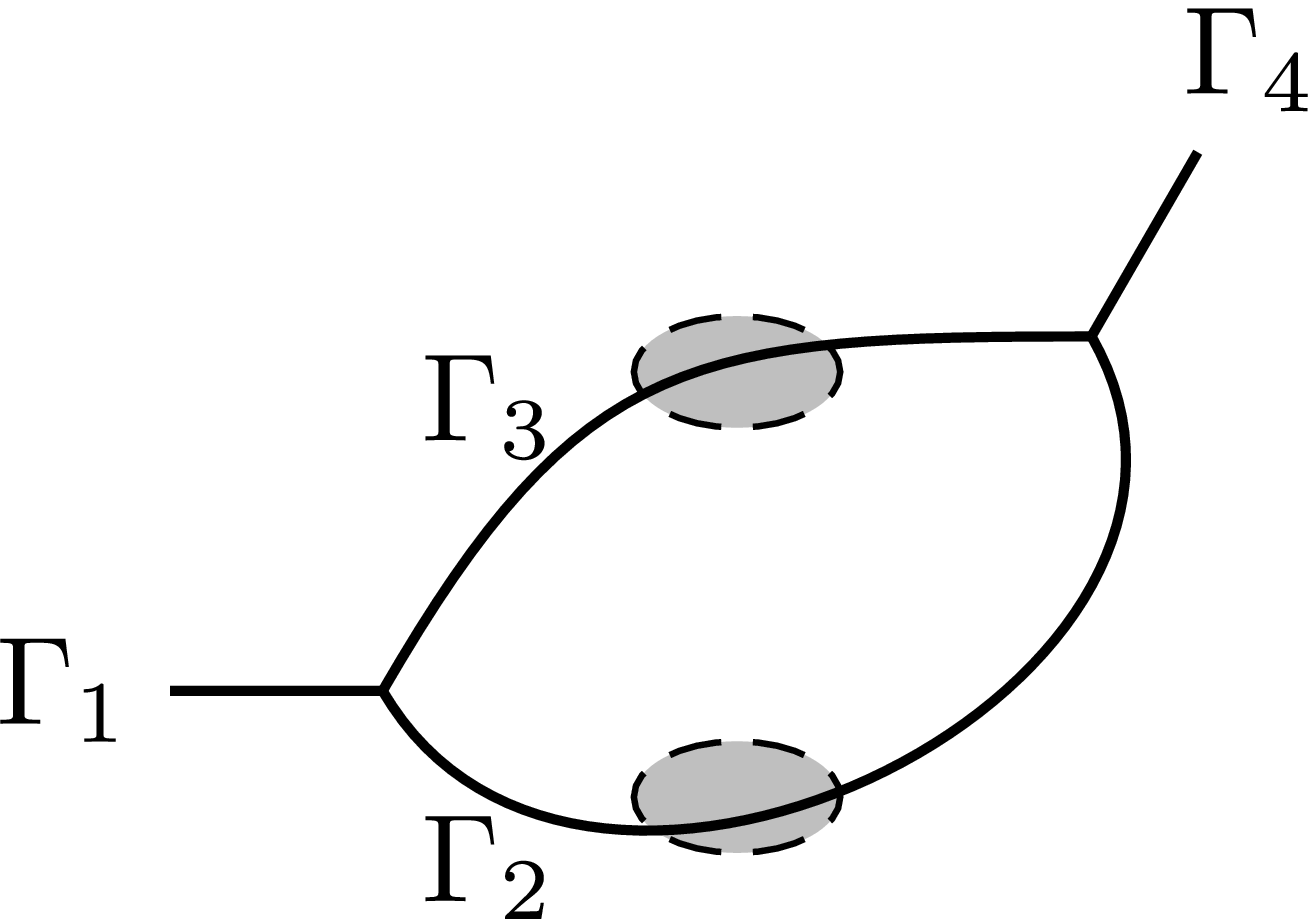}}&
\subfigure[Case (0,2)]{\label{f:(0,2)stack}\includegraphics[scale=.28]{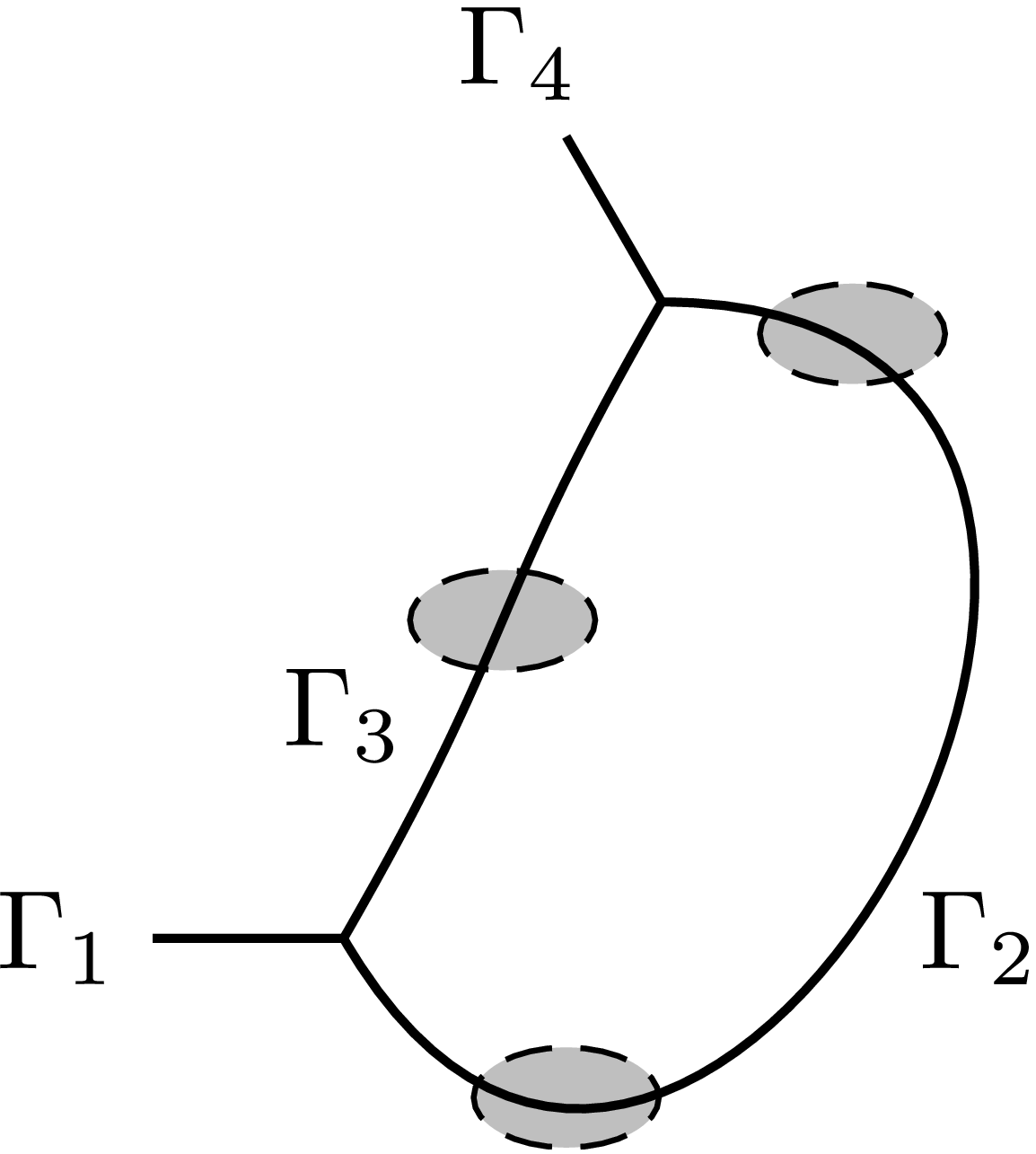}}&
\subfigure[Case (0,3)]{\label{f:(0,3)stack}\includegraphics[scale=.28]{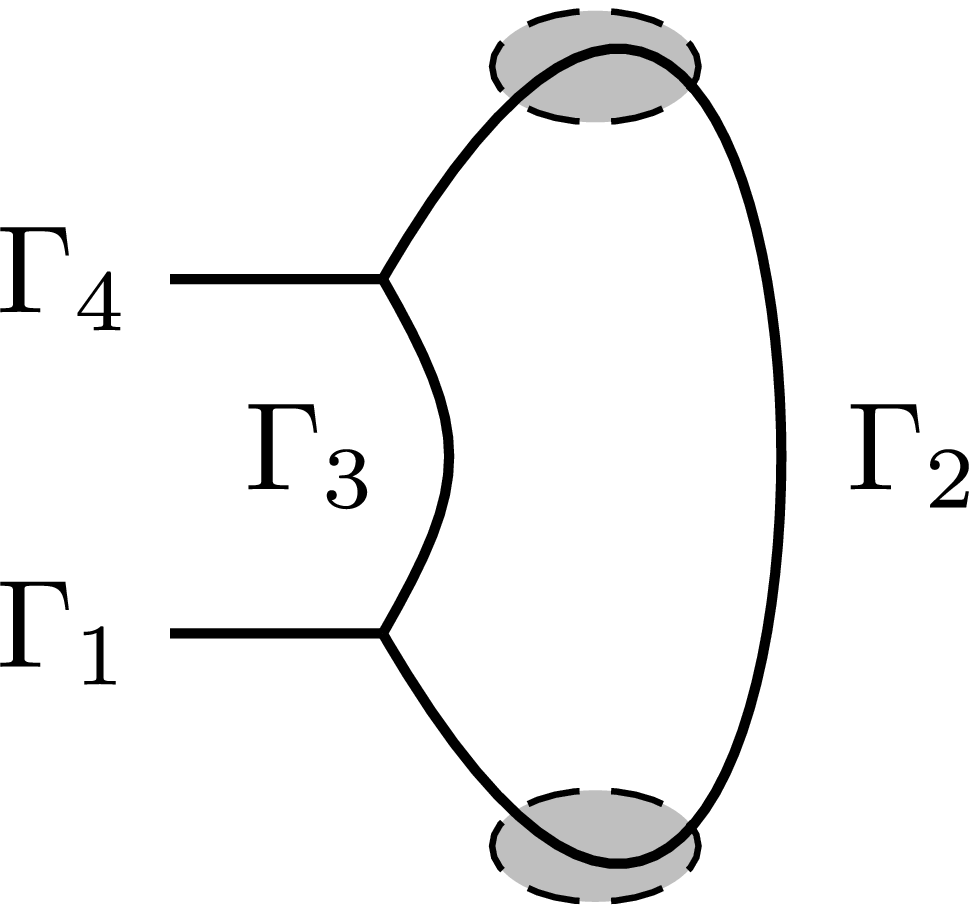}}\\
\subfigure[Case (0,-1)]{\label{f:(0,-1)stack}\includegraphics[scale=.28]{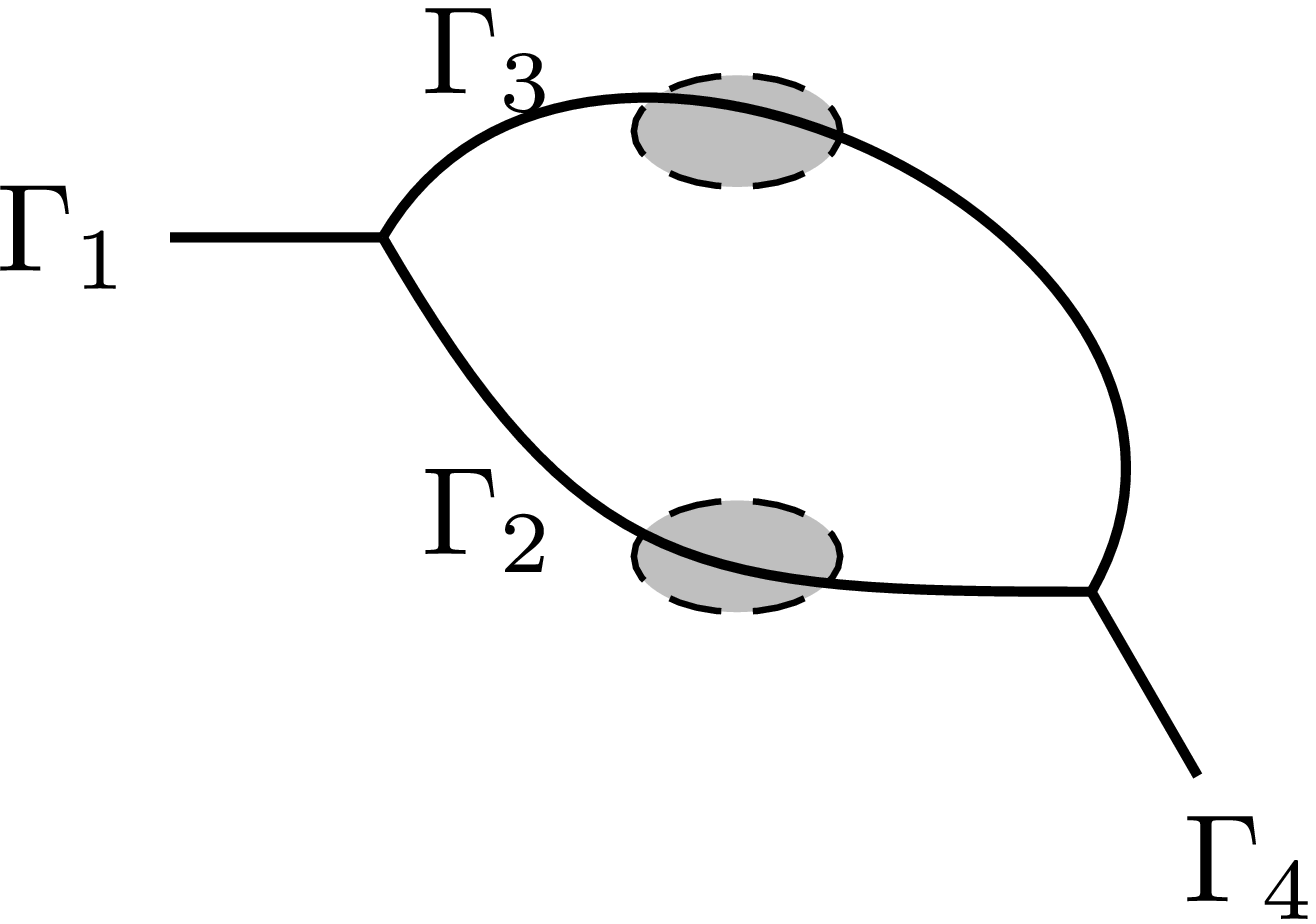}}&
\subfigure[Case (0,-2)]{\label{f:(0,-2)stack}\includegraphics[scale=.28]{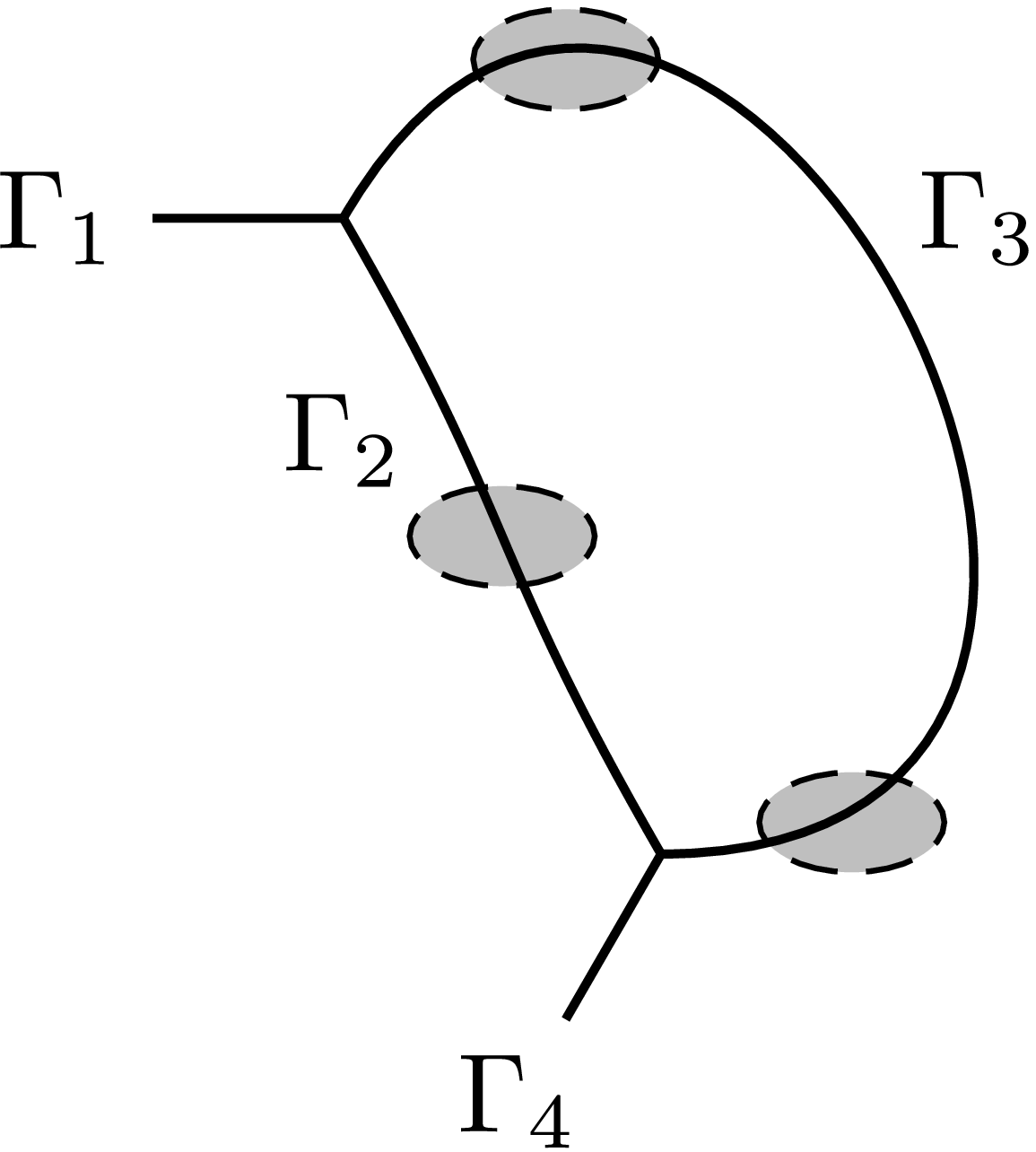}}&
\subfigure[Case (-1,-2)]{\label{f:(-1,-2)stack}\includegraphics[scale=.28]{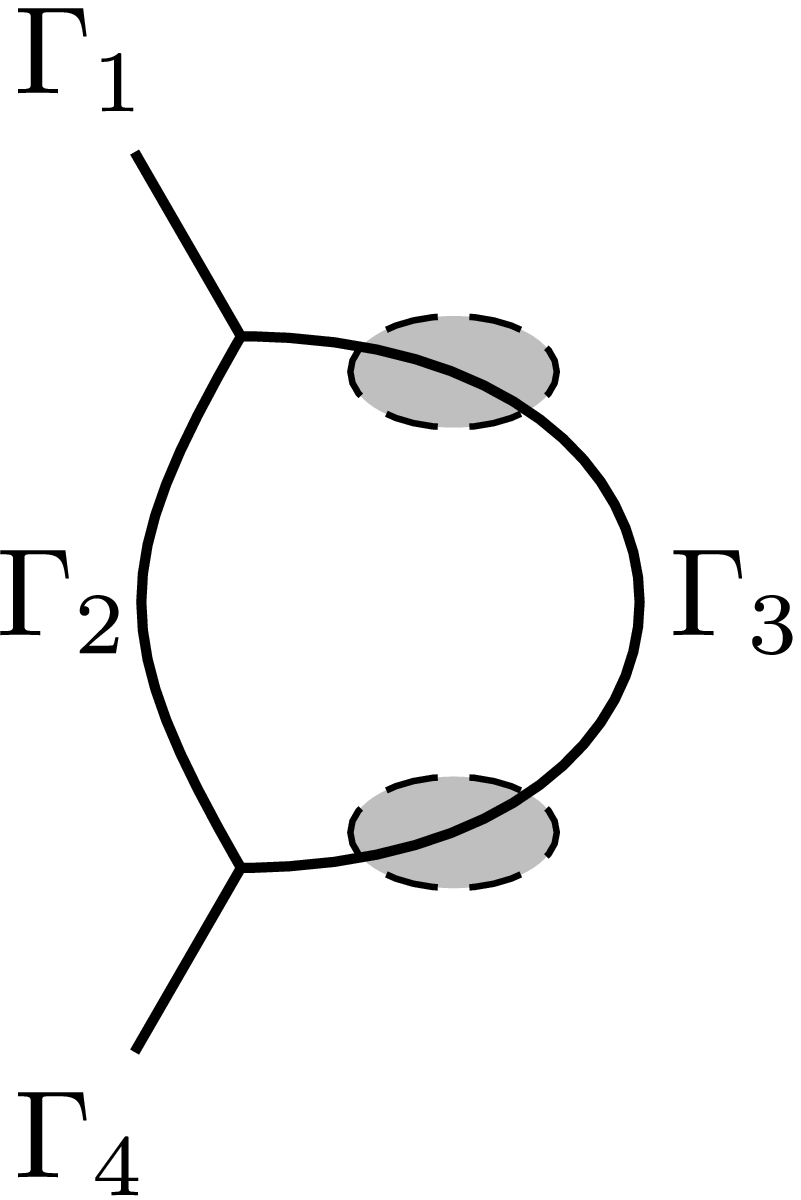}}\\
\subfigure[Case (-1,-1)]{\label{f:(-1,-1)stack}\includegraphics[scale=.28]{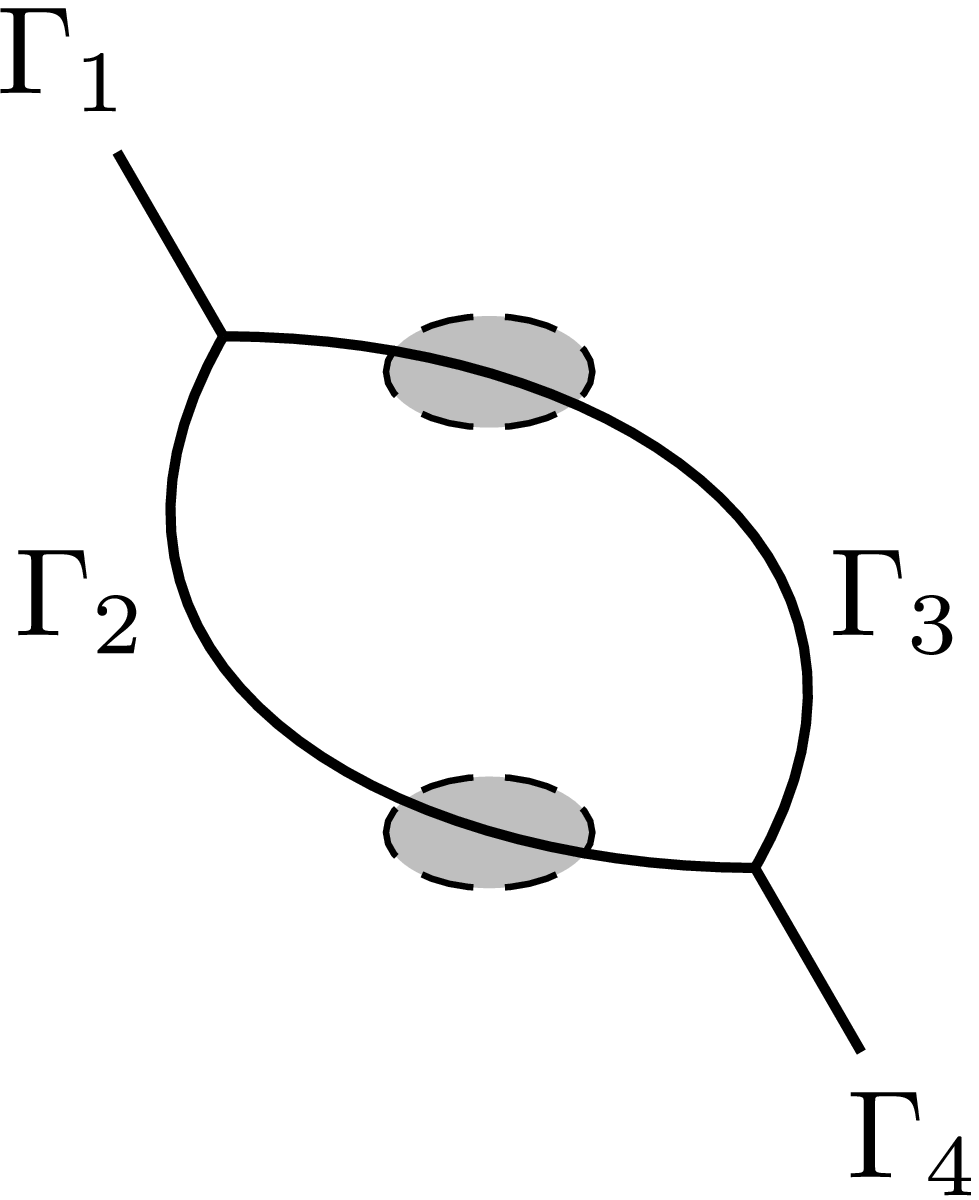}}&
\subfigure[Case (-1,1)]{\label{f:(-1,1)stack}\includegraphics[scale=.28]{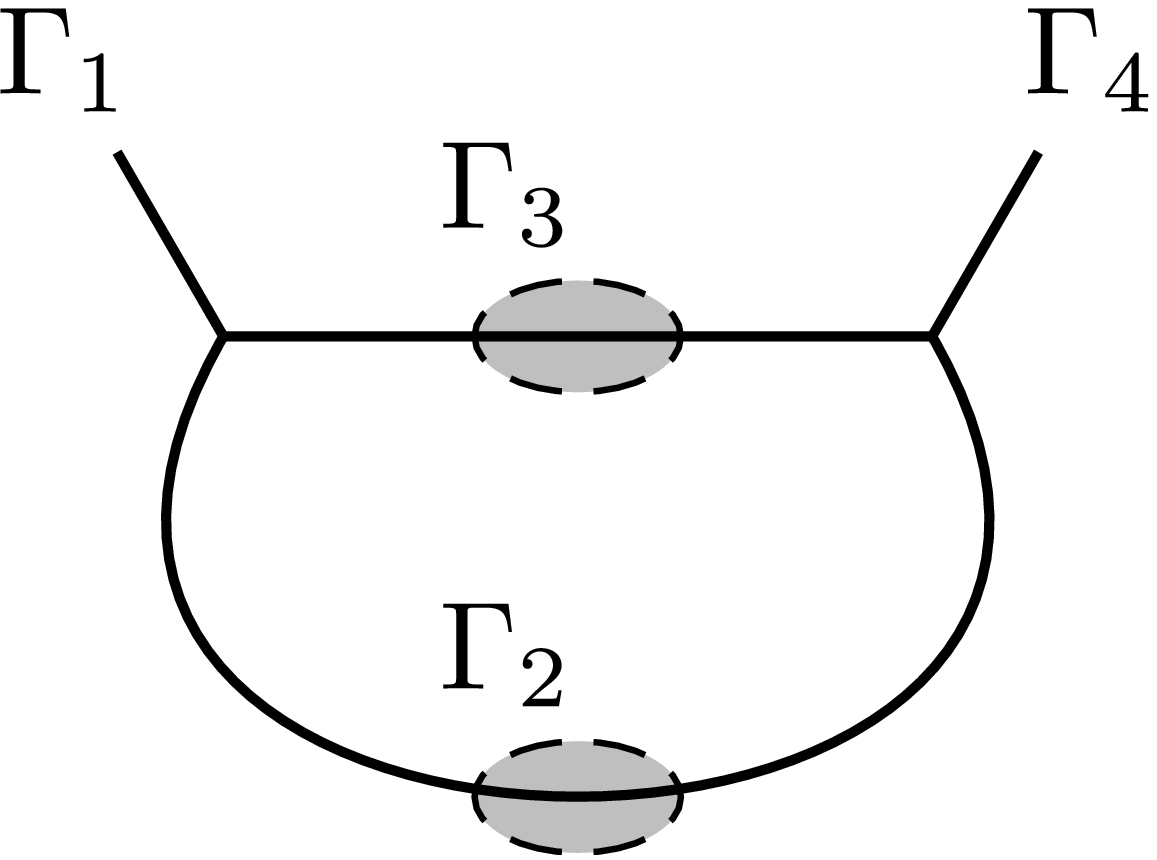}}&
\subfigure[Case (1,-1)]{\label{f:(1,-1)stack}\includegraphics[scale=.28]{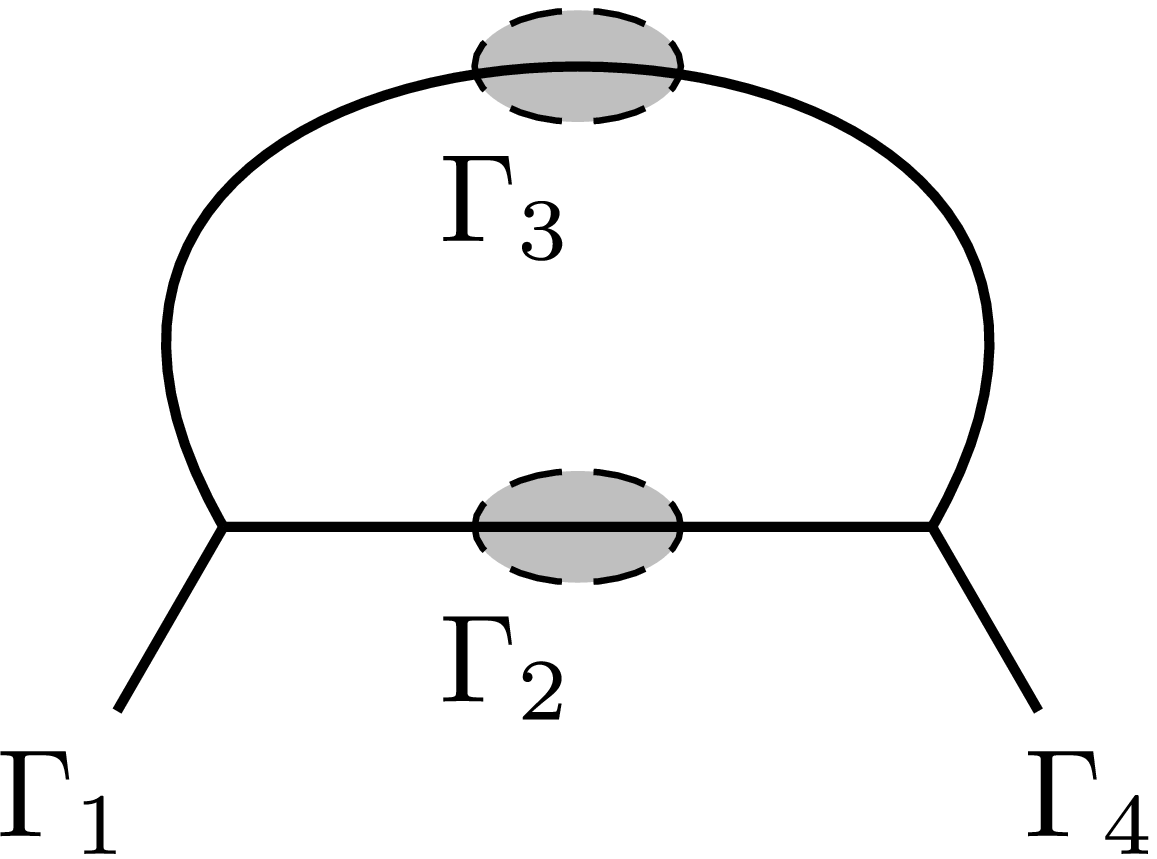}}
\end{tabular}
\caption{Different $(m_p,m_q)$ component cases we consider.  See also \figref{f:(0,0)stack}.  Case $(i, j)$ is symmetrical under relabeling to $(j-3, i-3)$, and under horizontal reflection to $(-j, -i)$.} \label{f:leafstackfigures}
\end{figure}

\smallskip\noindent$\bullet$ Case $(0,2)$: See \figref{f:(0,2)stack}.  (With the notation of \secref{s:notation}, this includes the case in which $\Gamma_3$ is a vertical hyperplane.)
A right-side-up, near-graph stack cannot be placed along $\Gamma_2$ without contradicting \corref{t:stackdown}, since $\Gamma_2$ turns downward past the vertical.  The presence of upside-down, near-graph leaves sitting beneath $\Gamma_2$ beyond where $\Gamma_2$ turns past the vertical will not affect our argument.

If $\Gamma_3$ is a single piece of arc, then by \corref{t:graphcorollary} on $\Gamma_3$, $\f_2(p) < \f_2(q)$.  (If $\Gamma_3$ is a vertical hyperplane, then \corref{t:graphcorollary} doesn't apply, but still $\f_2(p) < \f_2(q)$.)  However, $[-\infty, \f_2(q)) \subset \f(\Gamma_2)$, so $\Gamma_2$ has an internal separating set, contradicting \corref{t:downvertical} or \corref{t:separatingset} (if there is a leaf below $\Gamma_2$). 

In \figref{f:(0,2)stack}, $\Gamma_3$ is drawn with $\theta_{\Gamma_3} \in [\pi/6, \pi/2]$ between $p$ and $q$, so it may appear that a near-graph stack cannot be placed on $\Gamma_3$.  However, this is not the case; see, e.g., \figref{f:(0,2)stackonstack}.  We claim that even if there are one or more near-graph stacks along $\Gamma_3$, still $\f_2(p) < \f_2(q)$ so $\Gamma_2$ has an internal separating set.  

Letting $r \in \Gamma_2$ be the point where $\Gamma_2$ turns vertical ($\f_2(r) = \infty$), we see that $y(p) \leq y(r) < y(q)$; $q$ is above $p$.  
By \claimref{t:sameouterboundaries}, the pieces of $\Gamma_3$ leaving $p$ and leaving $q$ are parts of the same Delaunay hypersurface, up to horizontal translation.
Let $s$ be the first vertex $\Gamma_3$ reaches after leaving $p$; we claim that $\Gamma_3$ from $p$ to $s$ goes above $q$.  Indeed, since $s$ is the left vertex of a near-graph stack, it must be that $\theta_3(s) \in [-\tfrac{\pi}{2},\tfrac{\pi}{6}]$; in particular, $\theta_3(s) \leq \min \{ \theta_3(p), \theta_3(q) \}$.  If $\Gamma_3$ is an unduloid, then $\ddot{\theta}_3 < 0$ as $\Gamma_3$ rises (see Eq.~\eqnref{e:ddottheta}), so $\theta_3$ is more than $\min \{ \theta_3(p), \theta_3(q) \}$ from $p$ until it reaches the height of $q$.  If $\Gamma_3$ is a nodoid or a circle, then $\dot{\theta}_3 = -\kappa_3 < 0$, yielding the same conclusion.  

Therefore, $\Gamma_3$ from $p$ to $s$ goes above $q$; so there exists $q'$ along $\Gamma_3$ from $p$ to $s$, with $y(q') = y(q)$ and $\theta_3(q') = \theta_3(q)$.  Slide the ray $L_2(q)$ left horizontally to point $q'$; then \corref{t:graphcorollary} shows that $\f_2(q) - (x(q)-x(q')) > \f_2(p)$, so in particular $\f_2(q) > \f_2(p)$ as claimed. 

\begin{figure}
\includegraphics[scale=.28]{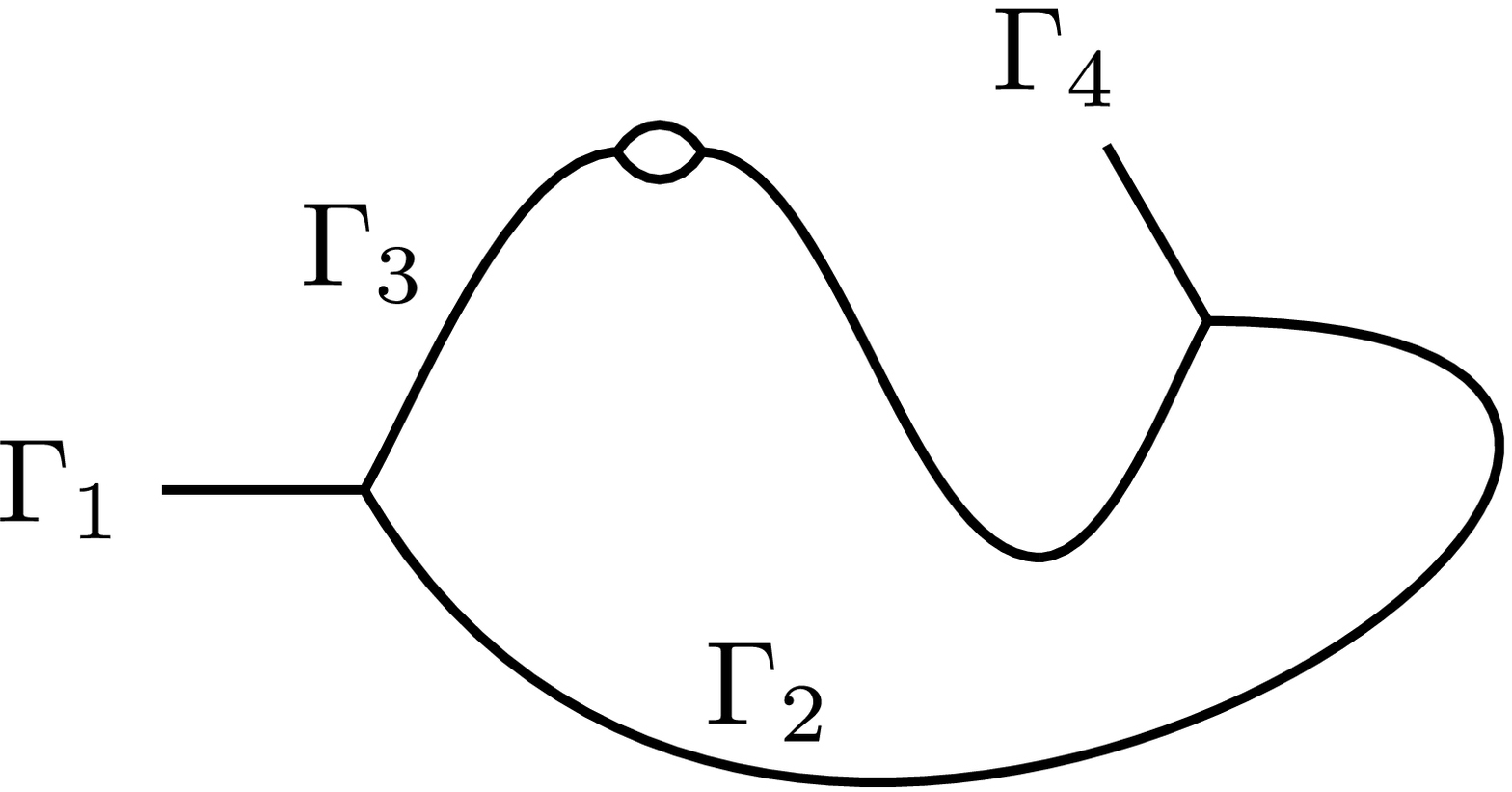}$\qquad$
\includegraphics[scale=.28]{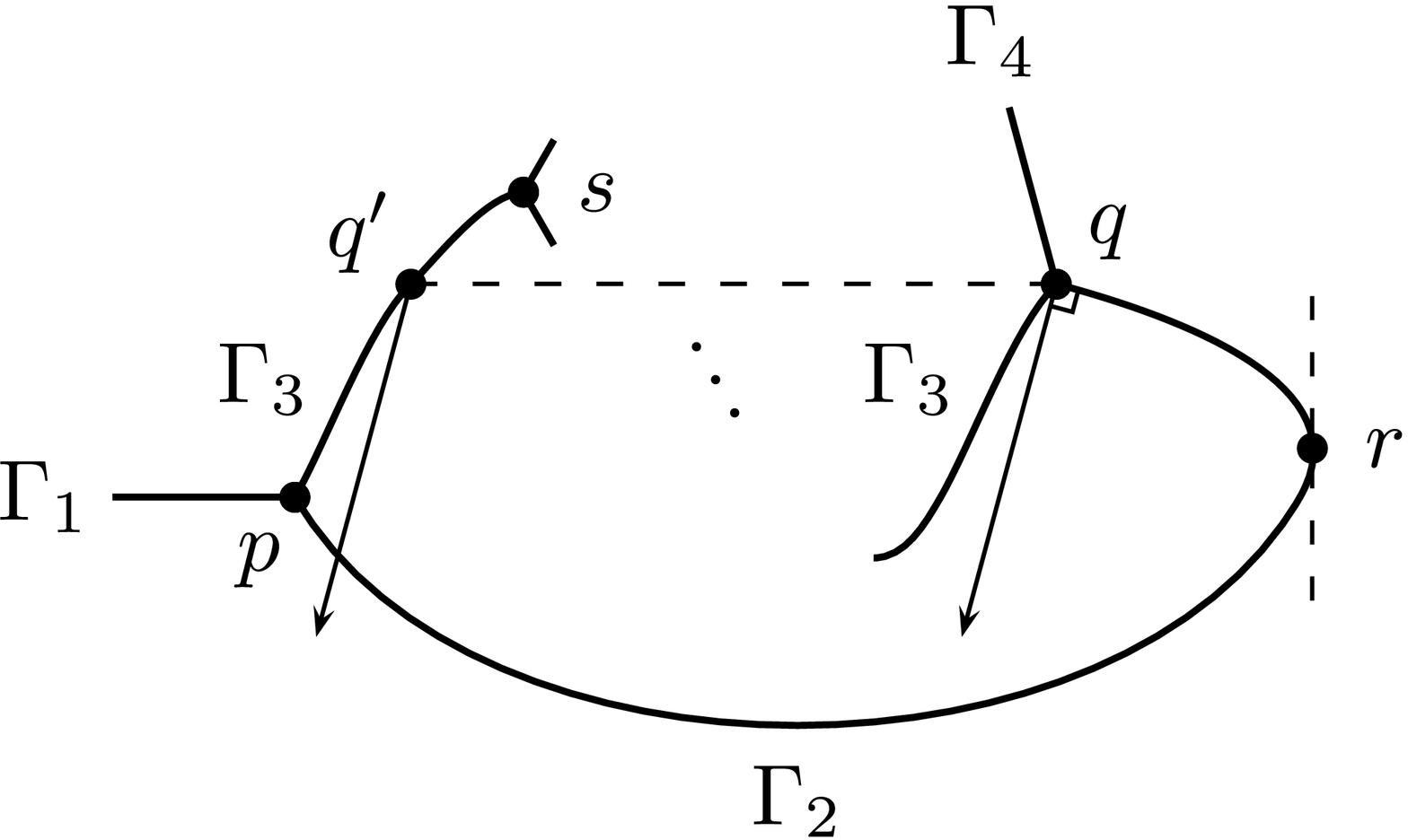}
\caption{
Left: A right-side-up, near-graph component stack might be placed on $\Gamma_3$ in \figref{f:(0,2)stack} if $\Gamma_3$ is an unduloid.  
Right: We show that nodoid $\Gamma_2$ has an internal separating set by horizontally 
translating the ray $L_2(q) = \protect \overrightarrow{q\, \f_2(q)}$ to the piece of arc leaving $p$, then applying \lemref{t:undularyray}.} \label{f:(0,2)stackonstack}
\end{figure}

\smallskip\noindent$\bullet$ Case $(0,3)$ or $m_q \geq 3$: See \figref{f:(0,3)stack}.  There cannot be a stack sitting above $\Gamma_2$ without contradicting \corref{t:stackdown}, since $\Gamma_2$ turns downward past the vertical.  

$\Gamma_3$ must be a convex-left nodoid (it cannot have any components along it by induction), so the ray $L_2(q) = \overrightarrow{q\, \f_2(q)}$ stays right of $\Gamma_3$, implying $\f_2(p) \in [-\infty, \f_2(q)) \subset \f(\Gamma_2)$, contradicting either \corref{t:downvertical} or \corref{t:separatingset} (if there is a leaf below $\Gamma_2$).

\smallskip\noindent$\bullet$ Cases $(0,-1)$, $(0,-2)$ and $(-1,-2)$: See Figures~\ref{f:(0,-1)stack}, \ref{f:(0,-2)stack} and~\ref{f:(-1,-2)stack}.  There cannot be a right-side-up near-graph stack above $\Gamma_3$ by \corref{t:stackdown} since $\Gamma_3$ turns downward past the vertical.  Upside-down, near-graph leaves cannot be placed along $\Gamma_2$ without contradicting \claimref{t:neargraphleafhigherpressure}.  The presence of upside-down, near-graph leaves sitting beneath $\Gamma_3$ beyond where $\Gamma_3$ turns past the vertical will not affect our argument.

Now $\f_1(p) \in [-\infty, \f_3(p)) \subset \f(\Gamma_3)$.  In case $(0,-1)$ or $(0,-2)$, $\f_2(p) < \f_1(p)$, while in case $(-1,-2)$, $[-\infty, \f_2(q)) \subset \f(\Gamma_2)$.  Regardless, \corref{t:separatingset} on $\Gamma_1$, $\Gamma_2$, $\Gamma_3$ implies $\f_2(q) \leq \f_1(p)$.  

In turn, this implies that $\Gamma_3$ must turn at least $\pi$ radians (i.e., $\theta_3(p) - \theta_3(q) \geq \pi$), and $\Gamma_2$ turns at most $\pi/3$ radians.  Indeed, this is necessary by definition in case $(0,-2)$.  In case $(0,-1)$, it follows since $q$ is to the right of $p$, although $\f_2(q) < \f_1(p)$.  In case $(-1,-2)$, it follows since the ray $L_1(p)$ stays above the convex-right nodoid $\Gamma_2$, and $\f_2(q) < \f_1(p)$.

Next, we argue that $\f_3(q) < \f_3(p)$.  Then since $[-\infty,\f_3(p)) \subset \f(\Gamma_3)$, there is an internal separating set in $\Gamma_3$, contradicting either \corref{t:downvertical} or \corref{t:separatingset} (if there is a leaf below $\Gamma_3$). 

Indeed, in case $(-1,-2)$, \corref{t:graphcorollary} applies to give $\f_3(q) < \f_3(p)$.  
In case $(0,-1)$ or $(0,-2)$, we may assume that $\Gamma_2$ does not go through a minimum and maximum between $p$ and $q$ -- see \figref{f:technicallemmaunduloid}.  Therefore, in case $(0,-2)$ or $(0,-1)$ with $\theta_2(q) < 0$, $\Gamma_2$ is strictly decreasing between $p$ and $q$.  \corref{t:unduloidcorollary} applies to give $\f_3(q) < \f_3(p)$.  

\begin{figure}
\includegraphics[scale=.28]{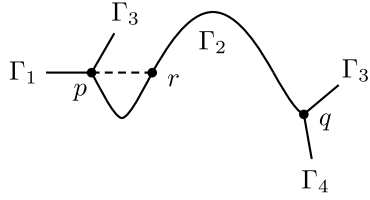}
\caption{In cases $(0,-1)$ and $(0,-2)$, if $\Gamma_2$ is a piece of unduloid, and goes through a minimum and maximum between $p$ and $q$, then there is a point $r \in \Gamma_2$ at the same height as $p$ with $\f(r) > \f_1(p)$.  Therefore $\f_1(p) \in \f(\Gamma_2)$.
} \label{f:technicallemmaunduloid}
\end{figure}

Finally, it remains to show that $\f_3(q) < \f_3(p)$ in case $(0,-1)$ with $\theta_2(q) \geq 0$.  \corref{t:unduloidcorollary} does not apply to $\Gamma_2$ in this case, so we need some additional geometry.\footnote{\cite[Lemma~5.9]{HutchingsMorganRitoreRos00dblbblR3} would suffice for this case, and in fact with \corref{t:unduloidcorollary} is enough for all of cases $(0,-1)$ and $(0,-2)$.  However, we will give a slightly simpler argument just for case $(0,-1)$ with $\theta_2(q) \geq 0$.}  By \lemref{t:simplegeometrylemma} (below) applied to point $q$ with $\phi = \theta_2(q) \in [0,\tfrac{\pi}{6})$ and applied to point $p$ with $\phi = \theta_1(p) \in (\theta_2(q), \tfrac{\pi}{6}]$ (also $y(p) > y(q)$), we get $\f_3(p) - \f_1(p) > \f_3(q) - \f_2(q)$.  Since $\f_2(q) \leq \f_1(p)$, $\f_3(q) < \f_3(p)$ as desired.

\smallskip\noindent$\bullet$ Case $(-1,-1)$: See \figref{f:(-1,-1)stack}.  There cannot be a right-side-up near-graph stack above $\Gamma_3$ by \corref{t:stackdown}, nor an upside-down near-graph leaf under $\Gamma_2$ by \claimref{t:neargraphleafhigherpressure}.  It remains to eliminate the case of a $(-1,-1)$ leaf -- we repeat the argument of \cite[Prop.~7.1]{ReichardtHeilmannLaiSpielman03dblbblR4}.  Now $\f_1(p) \in [-\infty, \f_3(p)) \subset \f(\Gamma_3)$.  The ray $L_1(p)$ stays left of the convex-right nodoid $\Gamma_2$, so $\f_1(p) < \f_2(q)$ and $\f_1(p) \in [-\infty, \f_2(q)) \subset \f(\Gamma_2)$.  This contradicts \corref{t:separatingset} for $\Gamma_1$, $\Gamma_2$, $\Gamma_3$.  

\smallskip\noindent$\bullet$ Case $(-1,1)$ or $m_q \geq 1$: See \figref{f:(-1,1)stack}.  $\Gamma_2$ goes twice vertical, contradicting \corref{t:twicevertical} or \corref{t:separatingset}.  

\smallskip\noindent This and symmetrical considerations concludes the argument if either $\Gamma_1$ or $\Gamma_4$ descends approaching $p$ or $q$, respectively.  The last remaining case is $(1,-1)$, in which both $\Gamma_1$ and $\Gamma_4$ ascend approaching $p$ or $q$.  

\smallskip\noindent$\bullet$ Case $(1,-1)$ or $m_q \leq -1$: See \figref{f:(1,-1)stack}.  $\Gamma_3$ goes twice vertical, contradicting \corref{t:twicevertical} or \corref{t:separatingset}.  
\end{proof}

\begin{lemma} \label{t:simplegeometrylemma}
Let $r$ be a point at height $y(r) > 0$ above the axis $L$.  Let $\phi \in [0,\tfrac{\pi}{6})$.  Drop lines from $r$ at angles $\phi$ and $\phi+\tfrac{\pi}{3}$ from the vertical (\figref{f:simplegeometrylemma}).  Then the difference in the $x$-coordinates of these lines' intersections with $L$, 
\[
y(r) \left( \tan(\phi+\tfrac{\pi}{3}) - \tan\phi \right) \enspace ,
\]
is a strictly increasing function in both $\phi$ and $y(r)$.
\end{lemma}
\begin{proof}
Simplifying, 
\[
\tan(\phi+\tfrac{\pi}{3}) - \tan\phi = \frac{\tan\tfrac{\pi}{3}}{\tfrac12+\cos(2\phi+\tfrac{\pi}{3})} \enspace ,
\]
and $\cos(2\phi+\tfrac{\pi}{3}) > \tfrac12$ and is decreasing for $\phi \in [0,\tfrac{\pi}{6})$.
\end{proof}

\begin{figure}
\includegraphics[scale=.28]{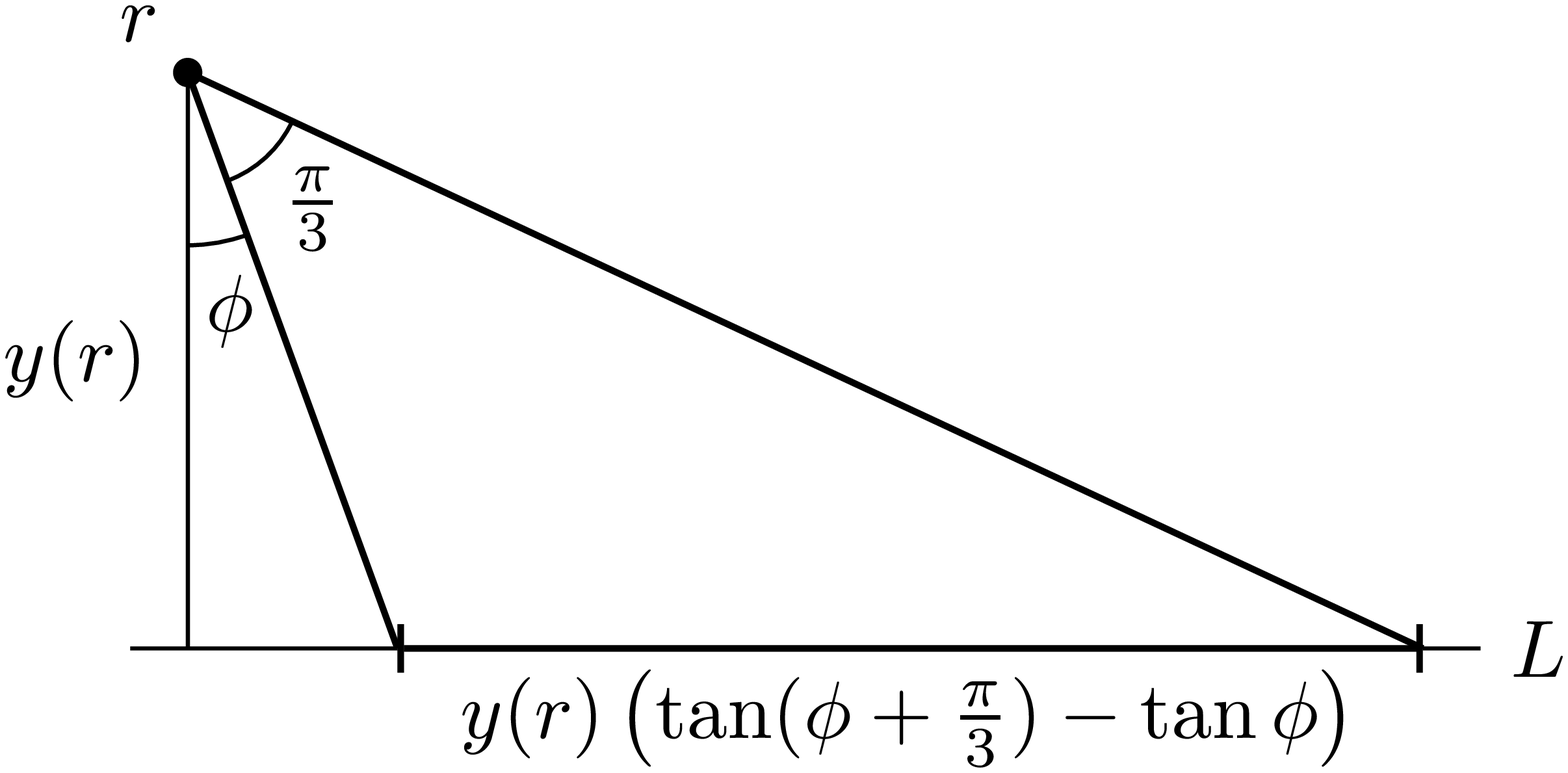}
\caption{\lemref{t:simplegeometrylemma}.} \label{f:simplegeometrylemma}
\end{figure}

\section{Root stability} \label{s:root}

The ``root'' of a nonstandard minimizing double bubble corresponds to the root of its associated tree of \thmref{t:structure} and \figref{f:structure}.  The root involves five arcs including two circular caps to either side, as in \figref{f:root}.  

\begin{figure} 
\includegraphics[scale=.28]{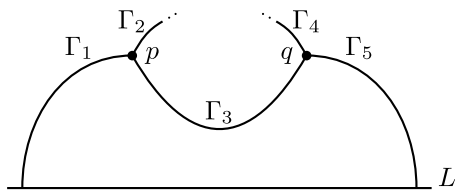}
\caption{The root component involves five arcs: $\Gamma_{2}$, $\Gamma_{3}$, $\Gamma_{4}$, and the two circular caps $\Gamma_{1}$ and $\Gamma_{5}$.} \label{f:root}
\end{figure}

\begin{proposition} \label{t:neargraph}
In a minimizer, the child component $C$ of the root component cannot be either an upside-down, near-graph leaf, or the base component of a right-side-up, near-graph stack.  
\end{proposition}

\begin{proof}
$\Gamma_1$ and $\Gamma_5$ are pieces of circles centered on the axis $L$, so $m_p$, $m_q$ the number of notches $p$, $q$ are rotated from their positions in \figref{f:root} must each lie in $\{-1,0,1\}$. 

In particular, neither $\Gamma_1$ nor $\Gamma_5$ can turn past the vertical to connect to an upside-down, near-graph leaf.  

For $C$ to be the base component of a right-side-up, near-graph stack, $(m_p, m_q) \in \{(0,0),(-1,0),(0,1)\}$.  In case $(-1,0)$, $\f_4(q) < \f_3(q)$ and $[-\infty,\f_3(q)) \subset \f(\Gamma_3)$, so there is a separating set across $\Gamma_3$, $\Gamma_4$, contradicting \corref{t:separatingset}.  Similarly, $(m_p, m_q) \neq (0,1)$.  

In case $(m_p, m_q) = (0,0)$, to avoid a separating set between $\Gamma_3$ and either $\Gamma_2$ or $\Gamma_4$, it must be that $\f(\Gamma_4) < \f(\Gamma_3) < \f(\Gamma_2)$, contradicting \lemref{t:neargraphstack} for the far endpoints of $\Gamma_2$ and $\Gamma_4$.  
\end{proof}

\section{Proof of the Double Bubble Conjecture}

\begin{proof}[Proof of \thmref{t:conjecture}]
Suppose that the minimizer is nonstandard.  Each region has a finite number of components by \corref{t:finite}.  Consider the minimizer's generating curves.  \thmref{t:componentclassification} implies that the root's child component is either an upside-down, near-graph leaf, or the base component of a right-side-up, near-graph component stack -- contradicting \propref{t:neargraph}.  Therefore, an area-minimizing double bubble must be the standard double bubble.  
\end{proof}

\begin{acknowledgements}
The author thanks Marilyn Daily for helpful discussions, and thanks Frank Morgan for comments on an early draft of the paper.  Research supported by NSF Grant PHY-0456720 and ARO Grant W911NF-05-1-0294. 
\end{acknowledgements}

\bibliographystyle{halpha}
\bibliography{bibliography}

\end{document}